\documentclass{qjmam}

\startpage{1}
\yr{2020}
\vol{1}
\issue{1}

\usepackage{amsmath,amsthm,amssymb,amsfonts, bm, bbm, stmaryrd}
\usepackage[colorlinks,citecolor=cyan,linkcolor=magenta]{hyperref}
\usepackage{graphics} 
\usepackage{subcaption} 
\usepackage{url} 
\usepackage{IEEEtrantools} 
\usepackage{tikz,pgfplots} 
\usetikzlibrary{arrows, arrows.meta}
\usetikzlibrary{patterns}
\usepackage{algorithm}
\usepackage[noend]{algpseudocode}
\usepackage{wrapfig}

\graphicspath{ {images/} }

\newtheorem{definition}{Definition}[section]
\newtheorem{lemma}[definition]{Lemma}
\newtheorem{assumption}[definition]{Assumption}
\newtheorem{Theorem}[definition]{Theorem}

\newtheorem{corollary}[definition]{Corollary}
\newtheorem{remark}[definition]{Remark}

\DeclareMathAlphabet\mathbit
\encodingdefault\rmdefault\bfdefault\itdefault
\DeclareOldFontCommand{\bi}{\normalfont\bfseries\itshape}{\mathbit}

\usepackage{mathrsfs}
\usepackage{xspace}

\newcommand{\cK}{\ensuremath{\mathcal K}\xspace}
\newcommand{\cG}{\ensuremath{\mathcal G}\xspace}
\newcommand{\cV}{\ensuremath{\mathcal V}\xspace}
\newcommand{\cW}{\ensuremath{\mathcal W}\xspace}

\newcommand{\be}{\begin{equation}}
    \newcommand{\ee}{\end{equation}}

\def\fakebold#1{\relax\ifvmode\leavevmode\fi%
    \ifmmode%
    \setbox0=\hbox{$#1$}%
    \else%
    \setbox0=\hbox{#1}%
    \fi%
    \kern-.02em\copy0 \kern-\wd0%
    \kern .04em\copy0 \kern-\wd0%
    \kern-.0125em\raise.02em\box0%
}%

\renewcommand{\geq}{\geqslant}
\renewcommand{\leq}{\leqslant}

\usepackage{paper_macro}

\begin{document}
    
    \title[Adaptive modelling of variably saturated seepage problems] 
    {Adaptive modelling of variably saturated seepage problems}

    
    \author[b.~ashby \etal] {B. Ashby$^*$, C. Bortolozo$^\dagger$, A. 
    Lukyanov$^{*,\mathsection}$ \and T. Pryer$\ddagger$}
    
    \address{$^*$Department of Mathematics and Statistics, University of 
    Reading, Reading, UK,
        \\
        $^\dagger$CEMADEN - National Center for Monitoring and Early Warning 
        of Natural Disasters, 
        General Coordination of Research and Development, S\~ao Jos\'e dos 
        Campos, Brazil
        \\
        $^\mathsection$P.N. Lebedev Physical Institute of the Russian Academy 
        of Sciences, Moscow 119991, 
        Russia
        \\
        $^\ddagger$Department of Mathematical Sciences, University of Bath, 
        Bath, UK.
    }
    
    \received{\recd --}

    \maketitle
    
    \eqnobysec
    
   \emph{ This article has been accepted for publication in The Quarterly 
   Journal of Mechanics and Applied Mathematics Published by Oxford 
   University Press.}
    
    \begin{abstract}
        {In this article we present a goal-oriented adaptive finite element
            method for a class of subsurface flow problems in porous media,
            which exhibit seepage faces. We focus on a representative case of
            the steady state flows governed by a nonlinear Darcy-Buckingham law
            with physical constraints on subsurface-atmosphere boundaries. This
            leads to the formulation of the problem as a variational
            inequality. The solutions to this problem are investigated using an
            adaptive finite element method based on a dual-weighted a posteriori
            error estimate, derived with the aim of reducing error in a specific
            target quantity.  The quantity of interest is chosen as volumetric
            water flux across the seepage face, and therefore depends on an a
            priori unknown free boundary.  We apply our method to challenging
            numerical examples as well as specific case studies, from which this
            research originates, illustrating the major difficulties that arise
            in practical situations. We summarise extensive numerical results
            that clearly demonstrate the designed method produces rapid error
            reduction measured against the number of degrees of freedom.}
    \end{abstract}

    \section{Introduction}
    
    The modelling of subsurface flows in porous media presents a multitude
    of mathematical and numerical challenges.  Heterogeneity in soils and
    rocks as well as sharp changes of several orders of magnitude in
    hydraulic properties around saturation are the multi-scale phenomena that
    are particularly difficult to capture in numerical models.  In
    addition, physically realistic domains include a wide variety of
    boundary conditions, some of which depend upon a free (phreatic)
    surface and therefore also upon the problem solution itself. These
    boundary conditions are described by inequality constraints. At points
    where the active constraint switches from one to the other, gradient
    singularities in the solution can arise which must be resolved well to
    avoid polluting the accuracy of the solution. The situation is
    analogous to a thin obstacle problem, for which gradient
    discontinuities arise around the thin obstacle
    \cite{milakis2008regularity}. For these reasons, such problems are
    good candidates for $h$-adaptive numerical methods, where a
    computational mesh is automatically refined according to an indicator
    for the numerical error.  It is the aim of such methods to provide the
    necessary spatial resolution with greater efficiency than is possible
    with structured meshes.
    
    A common model for steady flow in porous media in the geosciences is a
    free surface problem where the medium is assumed to be either
    saturated with flow governed by Darcy's law or dry with no flow at
    all. The free surface is the boundary between the two regions with a
    no-flow condition applied across it. Some authors solve this as a pure
    free boundary problem where the computational domain is unknown a
    priori such as in \cite{darbandi2007moving_phreatic_surface}. However, 
    this means that as the domain is updated, expensive re-meshing must 
    take place, allowing fewer of the data structures to be re-used from one 
    iteration
    to the next. To
    avoid the difficulties of this
    approach, in \cite{brezis1978nouvelle}, the problem formulation is
    modified to a fixed domain in which flow can take place (such as a
    dam) and the pressure variable defined on the whole domain, removing
    the need for changes in problem geometry and costly re-meshing during
    numerical simulations. The theory of this type of formulation is
    described in detail in \cite{oden1980theory}. A good approximation
    theory is available for finite element methods applied to such
    problems. It should be noted though that this model is a
    simplification, owing to the fact that it does not allow for
    unsaturated effects.
    
    To avoid the computational complexities of a changing domain, in this work 
    we consider the porous 
    medium to be variably saturated, and therefore we solve for pore pressure 
    over the entire domain 
    (cf \cite{zhang2001seepage_unified}). The results presented in 
    \cite{rulon1985development} 
    suggest that this approach is in fact necessary to accurately represent the 
    subsurface. It is also 
    expected that this framework will allow relatively easy extension to 
    unsteady cases where
    unsaturated effects are extremely important for the dynamics. 
    
    Although there has been much study of this problem, there are
    relatively few examples of adaptive finite element techniques being
    used. This is because the partial differential equation governing
    subsurface flow presents difficulties for the traditional theory of a
    posteriori estimation. This stems from the behaviour of the
    coefficient of hydraulic conductivity, which depends on the solution
    itself and approaches zero in the dry soil limit, leading to
    degeneracy of the PDE problem. This violates the standard assumption
    of stability in elliptic PDE problems.

    In an early work on the approximation of solutions to variational
    inequalities by the finite element method, Falk \cite{falk1974error}
    derives an a priori error estimate for linear finite elements on a
    triangular mesh when $N=2$ with $k(u) \equiv 1$, providing optimal
    convergence rates in the $H^1$-norm. The author also remarks that due
    to the relatively low regularity of the solution, higher order
    numerical methods can not provide a better rate. In situations such as this,
    local mesh refinement comes into its own.

    Traditional a posteriori estimation for finite element methods gives upper 
    bounds of the form 
    
    \begin{equation}
        \norm{u - u_h}_E 
        \leq 
        C \varepsilon (u_h, h, f)
    \end{equation}
    
    \noindent
    where $u$ is the exact solution to some partial differential equation, $C$ is a 
    positive constant, 
    $u_h$ 
    is the numerical solution, $h$ is the mesh function and $f$ 
    is problem data. $C$ is usually only computable for the simplest domains 
    and meshes, and can be 
    large. 
    The norm is usually an energy norm: a global measure chosen so that the 
    asymptotic 
    convergence rate of the method is optimal. In practical computations, 
    however, the user is often 
    not 
    interested in asymptotic rates that may never be reached, but would prefer 
    a sharp estimate of the 
    error to give confidence in the approximation. 
    
    The dual-weighted residual framework for error estimation was inspired by 
    ideas from optimal 
    control as a means to estimate the error in approximating a general quantity 
    of interest. Pursuing 
    this analogy, the objective functional to be minimised is the error in 
    numerically approximating a 
    solution to the PDE problem, the constraints are the PDE problem and 
    boundary conditions, and the 
    control variables are local resolution in the spatial discretisation.  
    
    There has been a huge amount of work on error estimation and adaptivity 
    using the dual-weighted 
    approach and it has shown to be extremely effective in computing quantities 
    which depend upon 
    local 
    features in steady-state problems in \cite{hartmann2002adaptive},  
    heterogeneous media 
    \cite{HoustonPorous2015goal} and variable boundary conditions in 
    variational inequalities
    \cite{Blum2000,suttmeier2008numerical}. In almost all cases the 
    performance of the goal based 
    algorithm cannot be 
    bettered in 
    efficiency. The goal-based framework also extends to time dependent 
    problems, where it has been 
    applied to the heat equation by \cite{parabolic_dual_vanderzee} and the 
    acoustic wave equation 
    by \cite{bangerth1999finite} among others.

    A common feature of numerical methods for seepage problems in the 
    literature is that they are 
    designed around getting a good representation of the phreatic surface, 
    namely the level set of zero 
    pressure head that divides saturated from unsaturated soil. There are 
    however many other possible 
    quantities of interest such as flow rate over a seepage face that could 
    represent the productivity of 
    a well.  In this work, correct representation of the phreatic surface is 
    prioritised only if it is 
    important for the calculation of the quantity of interest, and we let local 
    mesh 
    refinement do the work for us, rather than expensive re-meshing of the free 
    surface. Indeed, in the 
    current framework, mesh refinement is rather simple to implement and 
    relatively cheap.
    
    The dual-weighted residual method has been applied to linear problems with 
    similar characteristics. 
    In 
    \cite{Blum2000}, a simplified version of the Signorini problem is solved. 
    The authors of \cite{HoustonPorous2015goal} consider a groundwater flow 
    problem in which the 
    focus is to 
    estimate the error in the nonlinear travel time functional. In both cases, the 
    underlying PDE 
    operator is linear.
    
    The key step in deriving an a posteriori error bound for this variational 
    inequality is the introduction 
    of an intermediate function that solves the unrestricted PDE corresponding 
    to the inequality. This 
    allows the removal of the exact solution from the resulting bound. Finally, 
    the unrestricted solution 
    allows the problem data to enter into the problem, allowing a fully 
    computable a posteriori error 
    bound. In this paper, we apply these cutting edge techniques of a posteriori 
    error estimation and 
    adaptive computing to complex and relevant problems informed by 
    geophysical applications. We 
    demonstrate that the error bound is sharp and allows for highly efficient 
    error reduction in the target 
    quantity in a variety of situations which include geometric singularities, 
    multi scale effects in layered 
    media and complex boundary conditions at the seepage face.
    
    The remainder of the paper is set out as follows. In section
    \ref{sec:desciption_of_problem}, we describe the seepage problem and
    derive a weak formulation. The problem is discretised with a finite
    element method in section \ref{sec:fem}. Section \ref{sec:apost} is
    devoted to the derivation of a dual-weighted a posteriori estimate for
    the finite element error. Sections \ref{sec:implementation} and
    \ref{sec:algorithm} describe the particulars of the adaptive algorithm
    and our implementation of it. Section \ref{sec:numerics} contains
    numerical experiments, to illustrate the performance of the error
    estimate and adaptive routine in two test cases. Finally, section
    \ref{sec:case} contains the application of our adaptive routine to two
    case studies with experimental data chosen to illustrate some of the
    most difficult cases that arise in practice.

    \section{Description of Problem} \label{sec:desciption_of_problem}
    
    In this section, we give the mathematical formulation of the seepage
    problem and derive its weak form. Let $u$ denote the pressure head of
    fluid flowing in a porous medium in a bounded, convex domain $\Omega
    \subseteq \mathbb{R}^N$, $N=2$ or $3$ with boundary $\partial \W$. The
    flow of the fluid is described by the flux density vector
    $\mathbf{q}(u)$. Note that $\mathbf{q}(u)$ is not the fluid velocity
    $\mathbf{v}$, but is related to it by
    
    \begin{equation}\label{velocity}
        \mathbf{v} = \frac{\mathbf{q}(u)}{\phi},
    \end{equation} 
    
    \noindent
    where $\phi$ is the porosity of the medium, that is, the proportion of the 
    medium that may be 
    occupied by fluid. Flux density is related to the pressure field by
    \begin{equation} \label{Darcy_law}
        \mathbf{q}(u)
        :=
        - k(u) \nabla \left(u + h_z\right),
    \end{equation}
    
    \noindent
    where $h_z$ is the vertical height above a fixed datum representing
    the action of gravity upon the fluid and $k$ is a nonlinear function
    that characterises the hydraulic conductivity of the medium. We
    refrain from precisely writing $k$ here as our analytic results only
    require quite abstract assumptions on the specific form of $k$,
    however, for our practical tests, we will always have in mind that $k$
    is of van Genuchten type \cite{VanGenuch80}, compare with (\ref{k})
    and Figure \ref{fig:soil_parameters}. The modification of Darcy's law
    following the observation that hydraulic conductivity depends upon the
    capillary potential $u$ is due to \cite{buckingham1907studies}, and is
    a generalisation of the standard Darcy law that applies to soil that
    is completely saturated. In this case, the coefficient $k$ introduces
    strong nonlinearity into the problem.

    Now consider the steady state and suppose that $f$ is a source/sink
    term. Then we can combine \eqref{Darcy_law} with the mass balance
    equation
    
    \begin{equation} \label{conservation_of_mass}
        \nabla \cdot \mathbf{q}(u)
        =
        f
    \end{equation}
    
    \noindent
    to obtain the equation of motion for steady-state variably saturated flow
    
    \begin{equation}
        - \nabla \cdot k(u)\nabla (u + h_z)
        =
        f.
    \end{equation}
    
    To complete the above system and solve it, boundary conditions must be
    specified. We briefly review the most relevant here and point an
    interested reader to \cite{Beardynamics} for a more complete list.
    
    \emph{Boundaries that are in contact with a body of water} can be
    modelled by enforcing a Dirichlet boundary condition $u = g$,
    where $g$ is some function chosen based upon the assumption that
    the body has a hydrostatic pressure distribution. The boundary
    condition therefore enforces continuity of pressure head across
    the boundary. A hydrostatic condition can also be used to set the
    water table, and can represent the prevailing conditions far from
    the soil-air boundary.
    
    \emph{The flow of water across a boundary} is given by the component
    of the Darcy flux, \eqref{Darcy_law}, that is normal to the
    boundary. We will set $\mathbf{q}(u) \cdot \mathbf{n} = 0$ where
    $\mathbf{n}$ is the unit outward normal vector to $\partial \W$ to
    represent an impermeable boundary.
    
    \emph{At subsurface-air boundaries}, a set of inequality constraints
    must be satisfied. The pressure of water in the soil at such a
    boundary can not exceed that of the atmosphere, and when this pressure
    is reached, water is forced out of the soil, creating a flux out of
    the domain. The portion of a subsurface-air boundary at which there is
    outward flux is known as a seepage face, and it is characterised by
    the following conditions:
    
    \begin{equation} \label{seepage_face_condition}
        u \leq 0, \quad \mathbf{q}(u) \cdot \mathbf{n} \geq 0, \quad u 
        (\mathbf{q}(u) \cdot \mathbf{n}) = 0.
    \end{equation}

    \noindent
    We define the contact set to be the portion of the boundary along
    which the constraint $u \leq 0$ is active which is precisely the
    seepage face
    \begin{equation} \label{contact_set}
        B 
        :=
        \{ x \in \Gamma_A \mid u(x) = 0 \}.
    \end{equation}
    
    We are now ready to state the full problem. We divide the boundary of
    $\W$, $\partial\W$, into $\Gamma_A$, $\Gamma_N$ and $\Gamma_D$ 
    such
    that $\overline{\partial\W}
    =\overline{\Gamma_A}\cup\overline{\Gamma_N}\cup\overline{\Gamma_D}$.
     Here
    $\Gamma_A$ stands for the portion of the boundary at which a seepage
    face may form and $\Gamma_N$ and $\Gamma_D$ respectively denote
    portions of the boundary where it is known a priori that Neumann
    (respectively Dirichlet) boundary conditions are to be applied. The
    problem is to find $u$ such that

    \begin{IEEEeqnarray}{rC'l.l} 
        \nabla \cdot \mathbf{q}(u)
        :=
        - \nabla \cdot k(u)\nabla (u + h_z)
        &=
        f
        &\text{in} &\W  \label{darcy}
        \\
        \mathbf{q}(u) \cdot \mathbf{n} 
        & =
        0
        &\text{on} &\Gamma_N \label{flux_bc}
        \\
        u
        &=
        g
        &\text{on} &\Gamma_D \label{dirichlet_bc}
        \\
        u
        \leq
        0, \quad
        \mathbf{q}(u) \cdot \mathbf{n} 
        \geq
        0, \quad
        u \left( \mathbf{q}(u) \cdot \mathbf{n} \right)
        &=
        0
        &\text{on}\,\,& \Gamma_A, \label{seepage_face}
    \end{IEEEeqnarray}
    
    \noindent
    where $ f $ denotes a source/sink and $g = g(z)$ is an affine function
    representing hydrostatic pressure. We refer to
    figure \ref{schematic} for a visual explanation.
    
    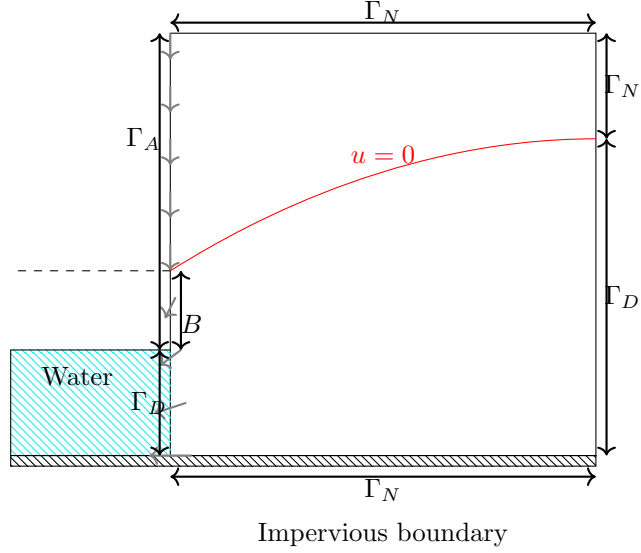
\begin{figure}[h!]
        \begin{center}
            \begin{tikzpicture}[scale=0.7] 
                \draw (0,0) -- (8,0) -- (8,8) -- (0,8) -- (0,0) -- cycle; 
                \draw [pattern=north west lines, pattern color=black] (-3,0) -- (8, 0) -- 
                (8, -0.2) -- (-3,-0.2) -- 
                cycle;
                \draw [pattern=north west lines, pattern color=cyan] (0,0) -- (-3, 0) -- 
                (-3, 2) -- (0,2);
                \draw[thick,gray,->] (0,8) -- (0,7.5);
                \draw[thick,gray,->] (0,7) -- (0,6.5);
                \draw[thick,gray,->] (0,6) -- (0,5.5);
                \draw[thick,gray,->] (0,5) -- (0,4.5);
                \draw[thick,gray,->] (0,4) -- (0,3.5);
                \draw[thick,gray,->] (0.1,3) -- (-0.1,2.6);
                \draw[thick,gray,->] (0.2,2) -- (-0.2,1.7);
                \draw[thick,gray,->] (0.3,1) -- (-0.3,0.8);
                \draw[thick,gray,->] (0.4,0) -- (-0.4,0.);
                \draw[thick,<->] (-0.2,2) -- (-0.2,8);
                \node at (-0.5,6) {$\Gamma_A$};
                \draw[dashed,-] (0,3.5) -- (-3,3.5);
                \draw[thick,<->] (0.2,2) -- (.2,3.5);
                \node at (0.4,2.5) {$B$};
                \draw[thick,<->] (-.2,2) -- (-.2,0);
                \node at (-.4,1) {$\Gamma_D$};
                \node[text=black] at (-1.75,1.5) {Water};
                \draw [red] (8, 6) parabola (0,3.5);
                \node[text=red] at (4, 5.7) {$u = 0$};
                \draw[thick,<->] (0.,-.4) -- (8.,-.4);
                \node at (4,-0.65) {$\Gamma_N$};
                \draw[thick,<->] (8.2,8) -- (8.2,6);
                \node at (8.5,7) {$\Gamma_N$};
                \draw[thick,<->] (8.2,6) -- (8.2,0);
                \node at (8.5,3) {$\Gamma_D$};
                \node at (4,-1.5) {Impervious boundary};
                \draw[thick,<->] (0.,8.2) -- (8.,8.2);
                \node at (4,8.4) {$\Gamma_N$};
            \end{tikzpicture}
        \end{center}
        \caption{A typical seepage problem. The upper part of the left lateral 
        boundary is in contact with 
            the atmosphere, while the lower part is underwater. The height at which 
            the level set $u=0$ meets 
            the boundary (marked with a dashed line) is a key unknown in seepage 
            problems.}
        \label{schematic}
    \end{figure}

    \subsection{Weak Formulation}
    
    In this section, we write the seepage problem \eqref{darcy} -
    \eqref{seepage_face} in weak form.  To that end, let $L^2(\W)$ be the
    space of square Lebesgue integrable functions defined on
    $\W$. Further, let $H^k(\W)$ be the space of functions whose weak
    derivatives up to and including order $k$ are also $L^2(\W)$. We then
    define the following function spaces:
    
    \begin{align}
        \cV^g &= \{v \in H^1(\W) \mid v = g \,\, \text{on}\,\, \Gamma_D \} \\
        \cK^g &=  \{v \in \cV^{g} \mid v \leq 0 \,\, \text{on}\,\, \Gamma_A \},
    \end{align}
    where boundary values are to be understood in the trace sense. Let $A$
    be a measurable subset of the domain $\W$, $v,w\in L^2(\W)$, then we
    write
    \begin{equation}
        (v \,,  w )_A := \int_A v\, w \diff x
    \end{equation} 
    as the $L^2(A)$ inner product. If the inner product is over $\W$, we
    drop the subscript and if $A$ is a subset of the boundary $\partial
    \W$, we interpret $(v \,, w )_A$ as a line integral.  
    
    We seek a weak solution $u \in \cK^g$ satisfying \eqref{darcy} -
    \eqref{seepage_face}. To that end, multiplying \eqref{darcy} by a test
    function $v \in \cK^0$ and integrating by parts, taking into account
    \eqref{flux_bc} gives
    
    \begin{equation}\label{weak_form}
        (\mathbf{q}(u), \,  \mathbf{n}\,v  )_{\Gamma_A}
        -
        ( \mathbf{q}(u), \, \nabla v  ) 
        =
        (f,  \, v )
        \quad \forall v \in \cK^0.
    \end{equation}
    
    \noindent
    By the boundary conditions and the definition of the space $\cK^0$,
    the boundary integral is negative so that \eqref{weak_form} can be
    written as:
    
    \begin{equation}\label{variational_inequality}
        (-\mathbf{q}(u),  \nabla v )
        \geq 
        ( f, \, v) 
        \quad \forall v \in \cK^0.
    \end{equation}
    
    \noindent
    We now extend the boundary data $g$ to a function $\bar{g} \in \cK^g$
    by insisting that $\bar{g} \equiv 0$ on $\Gamma_A$. We will address
    the choice of function $\bar{g}$ in Remark \ref{sec:barg} but for now it is
    sufficient to assume such a choice with this property exists. We may
    therefore set $ v = u-\bar{g} \in \cK^0$ in \eqref{weak_form} to give
    
    \begin{equation} \label{energy}
        (\mathbf{q}(u) , \mathbf{n}\,(u-\bar{g}))_{\Gamma_A} 
        -
        (\mathbf{q}(u), \,  \nabla (u-\bar{g})  )  
        =
        ( f, \, u-\bar{g}).
    \end{equation}
    
    \noindent
    Note that by \eqref{seepage_face} and the fact that $\bar{g}$ vanishes
    on $\Gamma_A$, the second term on the left hand side of \eqref{energy}
    is zero.  This result can be subtracted from
    \eqref{variational_inequality} to obtain the variational inequality in
    the standard and more compact form for such problems. The problem is
    then to seek $u\in\cK^g$ such that
    \begin{equation} \label{inequality_in_standard_form}
        \left(- \mathbf{q}(u), \nabla (v + \bar{g} - u) \right)
        \geq
        (f, v +\bar{g} - u)
        \quad \forall v \in \cK^0.
    \end{equation}
    
    In the seminal paper \cite{lions1967variational},
    existence and uniqueness of solutions is proved for problem
    \eqref{inequality_in_standard_form} in the case where $k(u) \equiv 1$,
    see also \cite{kinderlehrer1980introduction}. This is extendable to
    monotone nonlinear operators, however note the coefficient $k$ that
    parametrises the soil properties is often such that the operator does
    not satisfy this assumption, compare with Figure
    \ref{fig:soil_parameters}, although it can be regularised to mitigate
    this, as is done in for example \cite{BAUSE2004565}.
    
    In the case $k(u) \equiv 1$, the regularity result $u \in H^2(\W)$ is 
    established 
    \cite{brezis1978nouvelle}. To the author's
    knowledge, no such result is available for van Genuchten type
    nonlinearities. Indeed, our numerical results indicate this cannot be
    the case as the problem lacks regularity around the boundary of the
    contact set, shown in figure \ref{schematic} as the boundary between
    $B$ and $\Gamma_A \backslash B$.

    \section{Finite Element Method} \label{sec:fem}
    
    In this section, we introduce a finite element method to discretise
    \eqref{inequality_in_standard_form}. Let us assume that the domain
    $\W$ is polyhedral. Then we can define an exact subdivision of $\W$
    into a finite collection $\mathcal{T}$ of polygonal elements
    satisfying \cite[\S2]{ciarlet2002finite}.
    
    \begin{enumerate}
        \item 
        $K \in \mathcal{T}$ is an open simplex or open box, for example for 
        $N=2$, the mesh 
        would consist of triangles or quadrilaterals;
        \item 
        Two distinct elements intersect in a common vertex, a common edge or 
        not at all ($N=2$), and a 
        common vertex, edge or face or not at all ($N=3$);
        \item 
        $ \cup_{K \in \mathcal{T}} \overline{K} = \overline{\W}$.
    \end{enumerate} 
    
    \noindent
    We assume in addition that $\Gamma_A$ aligns with the mesh in the sense 
    that for all $K \in 
    \mathcal{T}$, $\partial K \cap \partial \W$ is either fully contained in 
    $\Gamma_A$ or else 
    intersects $\Gamma_A$ in at most one point ($N=2$) or one edge ($N=3$). 
    We make a similar assumption on elements lying on $\Gamma_D$. 
    For this choice of $\mathcal{T}$ we define the space
    \begin{equation}
        \cV^g_h
        =
        \{ v \in \cV^g \mid v \, \text{has total degree 1 on each} \, K \in 
        \mathcal{T} \}
    \end{equation}
    and the discrete subset
    \begin{equation}
        \cK^g_h
        =
        \{ v \in \cV^g_h \mid v \leq 0  \text{  on  } \Gamma_A \}.
    \end{equation}
    
    \noindent
    Note that for triangles or quadrilaterals when $N=2$ and tetrahedra
    and hexahedra when $N=3$, since a function $v_h \in \cK^0_h$ is linear
    along an element edge it is fully determined by its nodal values, that
    is, the set $\{ v_h(x) \mid x \,\text{ is a vertex of} \,\,
    \mathcal{T} \}$. Further, by the assumption that $\mathcal{T}$ aligns
    with $\Gamma_A$, it is enough to enforce $v_h(x) \leq 0 $ at this
    finite collection of points. This is not necessarily true for higher
    order finite elements, and for this reason we restrict our attention
    to those of total degree $1$.
    
    \begin{remark}[Choice of the function $\bar{g}$]
        \label{sec:barg}
        Now we are in a position to describe the construction of an
        appropriate extension $\bar{g}$ of $g$. We define the space
        \begin{equation}
            \cV^{g,
                \, 0} = \{v \in \cV^g \mid v = 0 \text{ on } \Gamma_A \}
        \end{equation}
        and corresponding finite element space
        \begin{equation}
            \cV^{g, \, 0}_h :=  \cV^g_h\cap \cV^{g, \, 0}
        \end{equation}
        and let $\bar{g}$ to be the solution to the following finite element
        problem: Find $\bar g\in \cV^{g, \, 0}_h$
        \begin{equation}
            (\nabla \bar{g}, \nabla v_h)
            =
            0
            \quad 
            \forall v_h \in \cV^{0, \, 0}_h
        \end{equation}
        \noindent
        $\bar{g}$ therefore has $H^1$ regularity over $\W$, satisfies the
        boundary condition on $\Gamma_D$ in the trace sense, and vanishes on
        $\Gamma_A$. We remark that this ensures also $\bar{g} \in \cK^g$. In
        the following sections as an abuse of notation, we will identify $g$
        with $\bar{g}$ to simplify the exposition.
    \end{remark}
    
    We are now ready to state the finite element approximation to this
    problem. We seek $u_h \in \cK^g_h$ such that
    \begin{equation} \label{discrete_problem}
        \left( -\mathbf{q}(u_h), \nabla (v_h + g - u_h) \right)
        \geq
        (f, v_h + g - u_h)
        \quad \forall v_h \in \cK^0_h.
    \end{equation}
    
    \section{Automated error control} \label{sec:apost}
    
    In this section we describe the derivation of an error
    indicator for the problem \eqref{darcy} - \eqref{seepage_face}. In
    doing so we make use of a \emph{dual problem} that is related to the
    linearised adjoint problem commonly used for nonlinear problems, but
    we keep only the zeroth order component of the linearisation. We then
    proceed in a similar manner to \cite{Blum2000}, where the authors
    consider a linear problem, to obtain a bound for the error in the
    quantity of interest.
    
    \subsection{Definition of Dual Problem}
    
    The definition of the dual problem is intervowen with the primal
    solution $u$ as well as the finite element approximation $u_h$. To
    begin, we define the discrete contact set as:
    \begin{equation} \label{discrete_contact_set}
        B_h 
        :=
        \{ x \in \Gamma_A \mid u_h(x) = 0 \}.
    \end{equation}
    
    \noindent
    We let
    
    \begin{equation}\label{definition_of_G}
        \cG = \{ v \in V \mid v \leq 0 \,\,\, \text{on} \,\,\, B_h  \,\,\, \text{and} 
        \,\,\, \int_{\Gamma_A}
        -\mathbf{q} (u) (v + 
        u_h) 
        \cdot 
        \mathbf{n} \diff S
        \leq 0 \},
    \end{equation}
    
    \noindent
    and suppose $J $ is a linear form whose precise structure will be
    discussed later, and let $z \in \cG$ be the solution to the following
    variational inequality:
    
    \begin{equation} \label{pseudo_linearised_dual}
        (k(u) \nabla(\varphi - z), \nabla z)
        \geq 
        J(\varphi - z)
        \quad \forall \varphi \in \cG.
    \end{equation}

    Application of duality arguments to derive error bounds in non-energy
    norms require assumptions of well-posedness on the dual problem which
    may not hold. Sharp regularity bounds on the dual problem with
    $k(u)\equiv 1$ were only recently proven in \cite{christof1754finite}
    by a non-standard choice of dual problem. Indeed, the authors prove
    bounds on the finite element error in the $L^4$ norm of optimal order,
    that is, order $h^{2 - \varepsilon}$ for any $\varepsilon \in (0, 1
    \slash 2)$ where $h$ is the mesh size. This motivates us to make the
    following assumption which we will use in the a posteriori analysis,
    the proof of which is currently the topic of ongoing research.
    
    \begin{assumption}[Convergence in $L^2$]
        \label{key_assumption}
        With $u$ solving \eqref{darcy} - \eqref{seepage_face} and $u_h$ as 
        defined in 
        \eqref{discrete_problem}, 
        there are constants $C >0$ and $s > 1$ such that 
        \begin{equation}
            \norm{u - u_h}_{L^2(\W)} 
            \leq
            C h^{s}.
        \end{equation}
    \end{assumption}

    \begin{definition}[Unrestricted solution]
        We define a function $U$ to be the solution of the elliptic problem
        analogous to problem \eqref{darcy}-\eqref{dirichlet_bc} but without
        the inequality constraint \eqref{seepage_face}. That is, $U \in
        \cV^g$ satisfies
        \begin{equation}\label{unrestricted_prob}
            (-\mathbf{q}(U) , \nabla w) 
            =
            (f, w) \quad \forall w \in \cV^0.
        \end{equation}
        The omission of a boundary term in the weak form indicates that $U$
        satisfies $\mathbf q(U)\cdot \mathbf n = 0$ on $\Gamma_A$.
    \end{definition}

    \subsection{Error Bound}
    
    Observe that by construction the function $z + u - u_h$ is a member of
    the set $\cG$. Indeed, by \eqref{seepage_face} we have $u \leq 0$ on
    $B_h$, by definition of $B_h$ and $\cG$ respectively we have $u_h=0$ and
    $z \leq 0$ on $B_h$.  We may therefore take $\varphi = z + u - u_h$ in
    \eqref{pseudo_linearised_dual} to obtain
    
    \begin{equation}\label{bound_1}
        J(u-u_h)
        \leq 
        (k(u) \nabla (u - u_h), \nabla z).
    \end{equation}
    
    \noindent
    Writing 
    
    \begin{equation}\label{nonlinear_difference}
        (k(u) \nabla(u-u_h), \nabla z)
        =
        (\mathbf{q}(u_h) - \mathbf{q}(u), \nabla z)
        - ((k(u) - k(u_h))\nabla (u_h + h_z), \nabla z),
    \end{equation}
    
    \noindent
    and expanding
    
    \begin{equation}\label{taylor}
        k(u) - k(u_h)
        =
        \int_0^1 k'(u_h + s (u - u_h))(u - u_h) \diff s,
    \end{equation}
    
    \noindent
    we note that with the a priori assumption \ref{key_assumption}, we can
    assume that the second term on the right hand side of
    \eqref{nonlinear_difference} is higher order in the error $u-u_h$ than
    the first term, and can therefore be neglected when the computation
    error becomes small. We will therefore focus on the first term in the
    following analysis.
    
    In the following lemmata, we prove bounds on differences between the
    functions $u$, $u_h$ and $U$.

    \begin{lemma}[Properties of the unrestricted solution]
        With $u$ the primal solution defined through
        (\ref{variational_inequality}), $u_h$ the finite element
        approximation to $u$ given by (\ref{discrete_problem}), and $U$
        the unrestricted solution defined in \eqref{unrestricted_prob}, we
        have, for any $v \in \cK^0$ and $v_h \in \cK_h^0$,
        
        \begin{equation}\label{alternate_weak_form}
            (\mathbf{q}(u) - \mathbf{q}(U), \nabla (v + g - u))
            \leq 0
            \quad \forall v \in \cK^0
        \end{equation}
        
        \noindent
        and 
        
        \begin{equation}\label{alternate_discrete_form}
            (\mathbf{q}(u_h) - \mathbf{q}(U), \nabla (v_h +g - u_h))
            \leq 0
            \quad \forall v_h\in \cK^0_h.
        \end{equation}
    \end{lemma}
    
    \begin{proof}
        We choose test functions $w = v+ g - u$ and 
        $w = v_h  + g - u_h$ respectively in \eqref{unrestricted_prob} where $v 
        \in \cK^0$ and 
        $v_h \in K_h^0$ are arbitrary to see that 
        
        \begin{equation} \label{unrestricted_1}
            (-\mathbf{q}(U), \nabla (v + g - u)) 
            = 
            (f, v + g - u)
            \quad \forall v \in \cK^0
        \end{equation}
        
        \noindent 
        and 
        
        \begin{equation} \label{unrestricted_2}
            (-\mathbf{q}(U), \nabla (v_h + g - u_h)) 
            = 
            (f, v_h + g - u_h)
            \quad \forall v_h\in \cK_h^0.
        \end{equation}   
        Subtracting \eqref{inequality_in_standard_form} from 
        \eqref{unrestricted_1} and 
        \eqref{discrete_problem} from \eqref{unrestricted_2}, we arrive at the 
        desired result.
        \qed
    \end{proof}
    
    \begin{definition}[Restricted solution set]
        We define the set
        \begin{equation}\label{definition_of_W}
            \cW^g_h = 
            \{ v \in \cV^g_h \mid v \leq 0 \, \text{ on } \, B_h  \}.
        \end{equation}
        Note that $\cW^g_h$ is a lightly smaller set than $\cK^g_h$, but that
        $u_h \in \cW^g_h$. This means that $u_h$ in fact satisfies
        \begin{equation}\label{alternate_discrete_form_W}
            (\mathbf{q}(u_h) - \mathbf{q}(U), \nabla (v_h + g - u_h))
            \leq 0
            \quad \forall v_h \in \cW^0_h.
        \end{equation}
    \end{definition}
    
    \begin{lemma}[Galerkin orthogonality]
        \label{galerkin}
        With $u$ the primal solution defined through
        (\ref{variational_inequality}) and $u_h$ the finite element
        approximation to $u$ given by (\ref{discrete_problem}) we have   
        \begin{equation}
            (\mathbf{q}(u_h) - \mathbf{q}(u), \nabla z_h)
            \leq 
            (\mathbf{q}(U)  - \mathbf{q}(u) , \nabla (z_h + u_h - u)) \quad \forall 
            z_h
            \in \cW^0_h,
        \end{equation}
        in analogy to the usual Galerkin orthogonality result.
    \end{lemma}
    
    \begin{proof}
        We can write
        \begin{equation} \label{useful_identity}
            \begin{split}
                (\mathbf{q}(u_h) - \mathbf{q}(u), \nabla z_h)
                & =
                (\mathbf{q}(U)  - \mathbf{q}(u), \nabla (z_h + u_h - u)) \\
                & +
                (\mathbf{q}(u_h) - \mathbf{q}(U), \nabla z_h) \\
                & + 
                (\mathbf{q}(U) - \mathbf{q}(u), \nabla (u - u_h)).
            \end{split}
        \end{equation}
        Now suppose $z_h \in \cW^0_h$. By setting $v_h = u_h + z_h - g$ in 
        \eqref{alternate_discrete_form}, the 
        second 
        term 
        on the right hand side of \eqref{useful_identity} is negative. Similarly, 
        choosing $v = u_h -g$ in 
        \eqref{alternate_weak_form}, the final term is also negative, and the 
        result follows. 
        \qed
    \end{proof}
    
    \begin{lemma}[Property of the dual solution]
        \label{ineq_0}
        Let $u$ be the primal solution defined through
        (\ref{variational_inequality}), $z$ be the dual solution from
        (\ref{pseudo_linearised_dual}) and $u_h$ the finite element
        approximation to $u$ given by (\ref{discrete_problem}). Then, we
        have
        \begin{equation}
            (\mathbf{q}(U) - \mathbf{q}(u), \nabla(z + u_h - u)) \leq 0.
        \end{equation}
    \end{lemma}
    
    \begin{proof}
        By the definition of $U$ we have 
        \begin{equation}
            (-\mathbf{q}(U), \nabla(z + u_h - u))
            = 
            (f, z + u_h - u)
        \end{equation}
        
        \noindent
        and by \eqref{weak_form},
        
        \begin{equation}
            (-\mathbf{q}(u), \nabla(z + u_h - u))
            = 
            (f, z + u_h - u)
            -  (\mathbf{q}(u), \mathbf{n} (z + u_h - u))_{\Gamma_A},
        \end{equation}
        
        \noindent
        and therefore, noting that $u (k(u) \nabla u) = 0$ on $\Gamma_A$,
        
        \begin{equation}
            \begin{split}
                (\mathbf{q}(U) - \mathbf{q}(u), \nabla(z + u_h - u))
                &=
                \int_{\Gamma_A} -\mathbf{q}(u) \cdot \mathbf{n} (z + u_h - u) \diff 
                S \\
                &=
                \int_{\Gamma_A} -\mathbf{q}(u) \cdot \mathbf{n} (z + u_h) \diff S 
                \leq 0,
            \end{split}
        \end{equation}    
        \noindent
        by the definition of the space $\cG$.
        \qed
    \end{proof}
    
    We now state the main result of this section.
    
    \begin{Theorem}[Error bound]
        \label{main_result}
        Let $u$ be the solution of \eqref{inequality_in_standard_form} and $u_h$ 
        the finite element 
        approximation to $u$. Let $U$ be the solution of the unrestricted 
        problem \eqref{unrestricted_prob}, $z$ the dual solution of 
        \eqref{pseudo_linearised_dual} and $z_h 
        \in W_h$ an arbitrary function. Then to leading order, we have
        
        \begin{equation} \label{main_result_eq}
            \begin{split}
                J(u - u_h)
                &\lesssim
                (\mathbf{q}(u_h) - \mathbf{q}(U), \nabla (z - z_h)) .
            \end{split}
        \end{equation}
    \end{Theorem}
    
    \begin{proof}
        Starting from \eqref{bound_1} and neglecting the higher order term, 
        justified by Assumption 
        \ref{key_assumption},
        
        \begin{equation} \label{ineq_1}
            \begin{split}
                J(u - u_h)
                & 
                \leq
                (\mathbf{q}(u_h) - \mathbf{q}(u), \nabla z)\\
                &
                = 
                (\mathbf{q}(u_h) - \mathbf{q}(u), \nabla (z-z_h))
                +
                (\mathbf{q}(u_h) - \mathbf{q}(u), \nabla z_h).
            \end{split}
        \end{equation}
        
        \noindent
        Combining with Lemma (\ref{galerkin}) gives
        
        \begin{equation} \label{more_steps}
            \begin{split}
                (\mathbf{q}(u_h) - & \mathbf{q}(u), \nabla (z-z_h))
                +
                (\mathbf{q}(u_h) - \mathbf{q}(u), \nabla z_h) \\
                \leq 
                & \,  (\mathbf{q}(u_h) - \mathbf{q}(u), \nabla (z-z_h))
                +
                (\mathbf{q}(U) - \mathbf{q}(u), \nabla (z_h + u_h - u)) \\
                =
                & \, (\mathbf{q}(u_h) - \mathbf{q}(u), \nabla (z-z_h))
                +
                (\mathbf{q}(U) - \mathbf{q}(u), \nabla (z + u_h - u))\\
                & +
                \, (\mathbf{q}(U) - \mathbf{q}(u), \nabla (z_h - z))
                \\
                =&
                (\mathbf{q}(u_h)- \mathbf{q}(U), \nabla (z - z_h))
                +
                (\mathbf{q}(U) - \mathbf{q}(u), \nabla(z + u_h - u)),
            \end{split}
        \end{equation}
        upon rearranging. The second term is negative by Lemma
        \ref{ineq_0}, completing the proof.  \qed
    \end{proof}
    
    To illustrate the usefulness of this result, we state the following corollary to 
    theorem 
    \ref{main_result}. 
    
    \begin{corollary}[A posteriori error indicator]    
        With the notation of theorem \ref{main_result}, we have the local
        error estimate
        \begin{equation} \label{element_contribution}
            J(u-u_h)
            \leq
            \sum_{K \in \mathcal{T}}
            (f - \nabla \cdot \mathbf{q}(u_h), z - z_h)_K
            + \frac{1}{2} (\jump{ \mathbf{q}(u_h) }, z-z_h)_{\partial K}.
        \end{equation}
    \end{corollary}
    
    \begin{proof}
        Since $U$ solves \eqref{unrestricted_prob}, we can replace it in the
        right hand side of \eqref{main_result_eq} and introduce the problem
        data:
        \begin{equation}
            (\mathbf{q}(u_h) - \mathbf{q}(U), \nabla (z - z_h))
            =
            (f , z - z_h) + (\mathbf{q}(u_h), \nabla (z - z_h)).
        \end{equation}
        After integrating by parts over each element we obtain the stated result.
        \qed
    \end{proof}

    Equation \eqref{element_contribution} gives a local quantity that we can 
    approximately evaluate to 
    give an estimate of the local numerical error. 
    Given a suitable approximation of the dual error $z - z_h$, this quantity can 
    be computed and used to 
    inform adaptive mesh refinement. The approximate computation of the 
    error estimate will be addressed 
    in section \ref{sec:algorithm}.
    
    \begin{remark}
        The analysis above allows the choice of $J$ to be made by the user 
        depending on the problem at 
        hand. The resulting estimate used in an adaptive algorithm will prioritise 
        the accurate computation of 
        $J$. For example,
        \begin{enumerate}
            \item  \ 
            Fix $x_0 \in \W$ and set $J_1(\varphi) = \varphi(x_0)$. An
            adaptive routine based upon the resulting estimate would
            prioritise accurate computation of the point value of the
            solution at $x_0$.
            \item  \
            Setting $J_2(\varphi) = (u - u_h, \varphi)$ would give an
            estimate of the error in the global error in $L^2$. Using
            suitable approximations, such an approach can be used in
            practice, see section 4 of \cite{becker2001optimal}.
            \item  \
            In seepage problems, a common quantity of interest is the
            volumetric flow rate of water through the seepage
            face. Since by definition the soil is saturated along the
            seepage face, the hydraulic conductivity takes the constant
            value $K_s$ (see section \ref{sec:implementation}). The
            fluid velocity is given by \eqref{velocity} and therefore
            the volumetric flow rate is given by
            \begin{equation} \label{functional}
                J(u) 
                :=
                -\int_{\Gamma_A} \frac{K_s}{\phi} \nabla(u + h_z) \cdot 
                \mathbf{n} \diff S
                = \int_{\Gamma_A} \frac{\mathbf{q}(u)}{\phi}\cdot \mathbf{n} 
                \diff S
            \end{equation}
        \end{enumerate}
    \end{remark}

    \section{Implementation Details} \label{sec:implementation}
    
    In this section we discuss various aspects of the practical solution of problem 
    \eqref{darcy} - 
    \eqref{seepage_face}. We first discuss the choice of parametrisation of $k$ 
    in \eqref{darcy}, then 
    present the iterative numerical algorithm used to solve the nonlinear 
    problem. Finally, we discuss 
    aspects of the adaptive routine and the tools required to approximately 
    evaluate the error estimate.
    
    \subsection{Hydrogeological Properties of the Medium}
    
    We make use of the popular model of \cite{mualem1976new} and
    \cite{VanGenuch80} to parametrise the unsaturated hydraulic properties
    of the soil. Consider a volume $V$ of a porous medium of total volume
    $V_{total}$. $V$ is made up of the solid matrix and air- or
    fluid-filled pores. If $V_{water}$ is the total volume of water
    contained in $V$, the volumetric water content $\theta$ is $V_{water}
    \slash V_{total}$, and therefore takes values between $0$ and the
    porosity of the soil. Point values of water content can be defined in
    the usual way of taking the water content over a representative
    elementary volume around the point (we refer to section 1.3 of
    \cite{Beardynamics} for details).  Water content is related to the
    pressure head in the soil, and can be modelled as a nonlinear function
    $\theta(u)$.  The dimensionless water content $\Theta$ was defined by
    van Genuchten as
    
    \begin{equation}
        \Theta(u) = \frac{\theta(u) - \theta_R}{\theta_S - \theta_R},
    \end{equation}
    
    \noindent
    where $\theta_R$ and $\theta_S$ are respectively the minimum and 
    maximum volumetric water 
    contents supported by a soil. Then the normalised water content $\Theta$ 
    takes values between 0 and 
    $1$ with $1$ corresponding to saturation. 
    Hydraulic conductivity, that is the nonlinear coefficient $k$ in 
    \eqref{darcy} is modelled similarly, and takes 
    strictly positive values reaching its maximum value at saturation. The 
    shapes of the functions $k$ and $\Theta$ 
    are dictated by 
    choice of dimensional parameters $K_S$ and $\alpha$, and 
    non-dimensional parameter $n$. The units 
    are $[K_S] = ms^{-1}$ and $[\alpha]=m^{-1}$. Soil Parameters are often 
    fitted 
    following laboratory experiments for a given soil. The saturated hydraulic 
    conductivity $K_S$ is the 
    maximum value that $k$ can take.  Finally, $\alpha$ and $n$ are shape 
    parameters whose physical 
    meaning is less clear. The 
    parameter $m$, introduced for ease of presentation, is defined by $m = (n-1) 
    \slash n$.
    This model has been shown to give good predictions in most soils near 
    saturation by 
    \cite{VanGenuch85describing}. 
    
    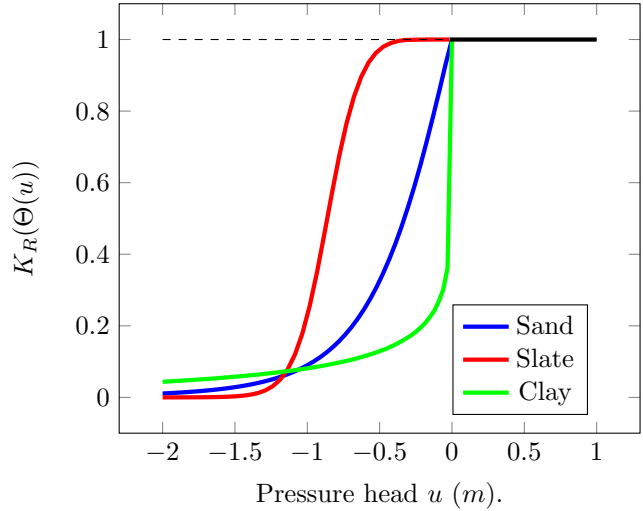
\begin{figure}[h!]
        \begin{center}
            \begin{tikzpicture}[
                declare function={
                    T(\u,\alpha,\n) =
                    (1/( (1+ (-\alpha*\u)^\n)^((\n-1)/\n)) )^0.5(
                    1 -  (    1 - (1/( (1+ (-\alpha*\u)^\n)^((\n-1)/\n)) )^(\n/(\n - 1))      
                    )^((\n-1)/\n)  )^2;
                },
                ]
                \begin{axis}[
                    xlabel={Pressure head $u$ ($m$).},
                    ylabel={$K_R(\Theta(u))$},
                    legend style={at={(0.9,0.3)}},
                    domain=-2:1,
                    ]
                    \addplot[blue, ultra thick,samples=70] {T(x,1,2.06)};
                    \addplot[red, ultra thick,samples=60] {T(x,0.982,6.97)};
                    \addplot[green, ultra thick,samples=100]{T(x,0.152,1.17)};
                    \addplot[dashed,black,samples=2] {1};
                    \addplot[black, ultra thick, domain=-0.015:1] {1};
                    \legend{
                        Sand,
                        Slate,
                        Clay
                    }
                \end{axis}
            \end{tikzpicture}
        \end{center}
        \caption{The permeability coefficient as a function of pressure head
            $u$ for different soil types. Note that $k(u) \to 0$ as $u \to
            -\infty$ but $K_R > 0$ for all $u$. Further, observe the
            smoothness of $K_R$ is quite different at $u=0$ for different soil
            types. This lack of regularity makes the numerical simulation of,
            say clay, particularly challenging. We also note that these
            functions are scaled by the saturated hydraulic conductivity,
            $K_S$, which varies enormously between different soils. The mean
            value for different soil types is $5 \times 10^{-6} ms^{-1}$
            (sand), $5 \times 10^{-9}ms^{-1}$ (slate) and $1 \times
            10^{-8}ms^{-1}$ (clay).}
        \label{fig:soil_parameters}
    \end{figure}
    
    \begin{equation}\label{VG_water_content}
        \Theta(u) =  \begin{cases}
            \frac{1}{\left[ 1 + (-\alpha u)^{n} \right] ^m} &\text{$u < 0$}\\
            1 &\text{$u \geq 0$}
        \end{cases}
    \end{equation}
    
    \begin{equation}\label{VG_relative_K}
        K_R(\Theta(u)) = \begin{cases}
            \Theta(u) ^{\frac{1}{2}} \left[ 1 - \left(1 - \Theta(u)^{\frac{1}{m}} 
            \right)^{m} \right] ^2 
            &\text{$u < 0$}\\
            1 &\text{$u \geq 0$}
        \end{cases}
    \end{equation}
    
    \noindent
    from which $k$ is then obtained by scaling by the saturated hydraulic 
    conductivity:
    
    \begin{equation}\label{k}
        k(u) = K_S \, K_R(\Theta(u)).
    \end{equation}
    
    Examples of hydraulic behaviour of different soils are shown in figure
    \ref{fig:soil_parameters}. The smoothness of the function $K_R$ as it
    approaches saturation is largely determined by the parameter $n$, with
    larger $n$ resulting in a smoother transition from unsaturated to
    saturated soil.
    
    \subsection{Solution Methods}
    
    To solve the nonlinear problem, we use a Picard iterative technique,
    common in the literature for computations in variably saturated flow
    \cite{paniconi1994newton_or_picard}. As described in
    \cite{seepage_scudeler2017examination}, we choose to implement the
    seepage face boundary condition using a type of active set strategy in
    a way that allows it to be updated within the Picard iteration during
    the solution process of the PDE. This has clear benefits for the
    accurate resolution of the seepage face, and it is especially
    important in the adaptive framework that the exit point be allowed to
    move to take advantage of increasing resolution during the adaptive
    process. A practical way of achieving this within the nonlinear
    iteration was first presented in \cite{neuman1973saturated}, but its
    focus on representing a single seepage face in an a priori assumed
    part of the boundary limits the range of applicability. The procedure
    was generalised in \cite{cooley1983seepage_face_algorithm} to allow
    any number of seepage faces by checking for unphysical behaviour at
    boundary nodes. This is essentially the method used here, but
    assignment is element-wise. Pressure and flux is checked along
    boundary faces which are then assigned as being on the seepage face or
    not, determining the boundary condition to be enforced at the next
    iteration. It was observed that this approach resulted in less
    oscillation of the exit point through the iterative process. This
    process can be thought of as a physically motivated version of a
    projection method for solving variational inequalities, as described
    in section 2 of \cite{oden1980theory}. The algorithm is illustrated
    below (see Algorithm \ref{nested_sol}).
    
    \begin{algorithm}
        \caption{An Iterative Scheme for the Seepage 
        Problem}\label{nested_sol}
        \begin{algorithmic}[1]
            \Require $u^0$, $\text{TOL}$, $N$
            \Ensure $u_h$, the approximation to the solution of the variational 
            inequality
            \State Set $i=1$;
            \While {$i < N$}
            \State Set $B_h := \{ x \in \Gamma_A \mid u_h^{i-1}(x) = 0 \}$;
            \For {degrees of freedom, $x_q$, over $\Gamma_A$}
            \If {$u_h^{i-1}(x_q) > 0$ and $x_q\notin  B_h$}
            \State Constrain $u_h^i(x_q) = 0$;
            \ElsIf {$(\mathbf{q}(u_h^i) \cdot \mathbf n)(x_q)  < 0$ and $x_q\in 
            B_h$}
            \State Constrain $(\mathbf{q}(u_h^i) \cdot \mathbf n)(x_q) = 0$;
            \Else 
            \State Leave boundary conditions unchanged;
            \EndIf
            \EndFor
            \State Find $u_h^i$ such that: $\int_\W k(u_h^{i-1}) \nabla(u_h^i + 
            h_z) \cdot \nabla v_h = \int_\W 
            f v_h$ for all $v_h$ over a space with boundary conditions as above;
            \If{$e:=\norm{u_h^i - u_h^{i-1}}_{L^2(\W)} < \text{TOL}$}
            \State Set $u_h:= u_h^i$;
            \State Break;
            \EndIf
            \State $i\texttt{++}$;
            \EndWhile
        \end{algorithmic}
    \end{algorithm}
    
    The nonlinear iteration is controlled by monitoring the difference in
    $L^2$-norm between successive iterates normalised by the norm of the
    newest iterate. Since we are concerned with the error in the finite
    element approximation, a very small iteration tolerance is set to
    ensure that the nonlinear error is small compared to discretisation
    error. The iteration registers a failure if this tolerance is not met
    within a specified number of steps, but in practice this did not
    occur.
    
    \subsection{Adaptive Algorithm} \label{sec:algorithm}

    In this section we describe the structure of the algorithm used to
    optimise the mesh,
    SOLVE$\rightarrow$ESTIMATE$\rightarrow$MARK$\rightarrow$REFINE.
    
    \begin{enumerate}
        \item 
        SOLVE the discretisation on the current mesh;
        \item 
        Calculate the local error ESTIMATE $\eta_k$;
        \item 
        Use $\eta_k$ to MARK a subset of cells that we wish to refine or
        coarsen based on the size of the local indicator;
        \item 
        REFINE the mesh.
    \end{enumerate}
    
    \subsubsection{Marking}
    
    Cells are marked for refinement using D\"{o}rfler marking, which was used 
    in \cite{Dorf1996} to 
    guarantee error reduction in adaptive approximation of the solution to the 
    Poisson problem. Choose 
    $\theta \in (0, 1)$. The estimate for the error 
    is given by $\eta = \sum_{K \in \mathcal{T}} \eta_K$. We 
    mark for refinement all elements $K \in \mathcal{M}$, where 
    $\mathcal{M}$ is a minimal collection 
    of 
    elements such that 
    
    \begin{equation}
        \sum_{K \in \mathcal{M}} \eta_K 
        \geq 
        \theta \eta.
    \end{equation}
    
    \subsubsection{Refining and Coarsening}
    
    An initial mesh $\mathcal{T}^0$ is generated over the computational 
    domain. In what follows, we use a 
    quadrilateral mesh since it allows for efficient refinement as detailed below. 
    During the solution 
    process, $\mathcal{T}^{l+1}$ is 
    obtained from $\mathcal{T}^l$ by adapting the mesh so that the local mesh 
    size is smaller around 
    cells marked for refinement and larger around cells marked for coarsening.  
    If an element is marked for 
    refinement it is 
    quadrisected. Thus, existing degrees of freedom do not need to be moved 
    meaning that the 
    change of mesh is rather efficient. Moreover we have a guarantee that the 
    shape of the elements 
    will not degenerate as the mesh is refined. Hanging nodes are permitted, but 
    constrained so that 
    the resulting discrete solution remains 
    continuous. It is therefore advantageous to allow a small amount of mesh 
    smoothing such as 
    setting a 
    maximum difference of grid levels between adjacent cells. In the 
    implementation of this algorithm, 
    the actual refinement and coarsening algorithm enforces 
    additional constraints to preserve the regularity of the mesh. For example, 
    the difference in 
    refinement level across a cell boundary is allowed to be at most one. In 
    practice, this is achieved by 
    refining some extra cells that were not marked to \lq smooth' the mesh. The 
    motivation behind this 
    is that many results on the approximation properties of finite element 
    methods require a degree of 
    mesh regularity. For a more detailed explanation of the implementation of 
    mesh refinement, we 
    refer the reader to the extensive \texttt{deal.ii} documentation available 
    online \cite{dealII91}. With 
    regards to 
    coarsening, due to the hierarchical structure of the meshes that result from 
    this process, cells that 
    have been refined \lq parent', that is, a quadrilateral in $\mathcal{T}^i$ for 
    some 
    $i \leq l$ in which it is fully contained. If all four \lq children' elements are 
    marked for coarsening, 
    the vertex at the middle of the four elements is removed and the parent cell 
    is restored resulting in 
    a locally coarser mesh.

    \subsubsection{Evaluating the Estimate}
    
    Recall the error estimate of proposition \ref{main_result}:
    
    \begin{equation}
        \eta 
        =
        \sum_{K \in \mathcal{T}} \eta_K,
    \end{equation}
    where
    \begin{equation}
        \eta_K
        =
        (f - \nabla \cdot \mathbf{q}(u_h), z - z_h)_K 
        +
        \frac{1}{2} (\jump{\mathbf{q}(u_h) }, z-z_h)_{\partial K}.
    \end{equation}
    
    \noindent
    Note that $\eta_K$ can only be approximately calculated since the
    exact dual solution $z$ is not available. There are several strategies
    for doing this which produce similar results
    \cite{bangerth2013adaptive}.  For computational efficiency, we choose
    a cheap averaging interpolation to obtain a higher order approximation
    of the dual solution as follows.
    
    The dual problem is solved on the same finite element space as the
    primal problem to obtain an approximation $z_h$. A function
    $\bar{z}_h$ is then constructed from $z_h$ in the following
    manner. Consider the mesh $\bar{\mathcal{T}^l}$ such that refining
    every element of $\bar{\mathcal{T}^l}$ produces $\mathcal{T}^l$. The
    nodal values of $z_h$ are used to produce a piecewise quadratic
    function on $\bar{\mathcal{T}^l}$. This technique is sometimes used as
    a post-processor to improve the quality of finite element
    approximation itself \cite{dedner2019residual}.  We make the
    approximation
    \begin{equation} \label{eta_k}
        \eta_K
        \approx
        (f - \nabla \cdot \mathbf{q}(u_h), \bar{z}_h- z_h)_K
        +
        \frac{1}{2} (\jump{\mathbf{q}(u_h)}, \bar{z}_h - z_h)_{\partial K}.
    \end{equation}
    
    \begin{remark}[Approximation of the space $\cG$]
        We finally remark that in the practical implementation, we must
        solve the dual problem in the set $\cW_h^g$ which may or may not be
        a subset of $\cG$. This is due to the fact that the exact contact
        set is not available, and so we do not have access to $\cG$. In
        fact, the authors of \cite{Blum2000} further suggest approximating
        $\cG$ by $\cG^0 : = \{ v \in \cV^0 \mid v = 0 \,\,\, \text{on}
        \,\,\, B_h \}$, and we also take this approach. This reduces the dual
        problem to a linear elliptic PDE, thereby simplifying the adaptive
        process.
    \end{remark}
    
    \section{Numerical Benchmarking} \label{sec:numerics}
    
    In this section, we present numerical results to demonstrate the
    effectiveness of the error estimate and adaptive routine in a range of
    realistic situations of interest in the analysis of subsurface
    flow. In this sense, we aim to benchmark our work to justify its use
    in the next section where we tackle specific case studies.
    
    All simulations presented here are conducted using \texttt{deal.II},
    an open source C\texttt{++} software library providing tools for
    adaptive finite element computations \cite{dealII91}. A fifth order
    quadrature formula is used in the assembly of the finite element
    system for each linear solve to attempt to capture some of the
    variation in the coefficients. To avoid any possible issues with
    convergence of linear algebra routines, an exact solver, provided by
    UMFPACK, is used to invert the system matrix.
    
    In all simulations we take as our quantity of interest the volumetric
    flow rate of water through the seepage face given in equation
    \eqref{functional}.
    
    \subsection{Example 1: Aquifer Feeding a Well}
    
    As a first two-dimensional example, let $\W = [0,1]^2$ represent a
    vertical section of a subsurface region. Spatial dimensions are given in 
    metres. 
    We refer to Figure
    \ref{schematic} for a visual representation of this problem, and give
    the specifics here.  The upper surface $\{(x,z) \mid z = 1\}$
    represents the land surface while $\{(x,z) \mid z = 0\}$ is
    impermeable bedrock. In both cases no-flux boundary conditions are
    enforced. We remark that in certain cases the land surface could
    exhibit seepage faces, as we will see in Example 2, but we assume that
    this will not be the case here. On $\{(x,z) \mid x = 1\}$, a
    hydrostatic Dirichlet condition is enforced for the pressure with the
    water table height set at 0.8$m$, that is, we set $u = 0.8 - z$ along
    this portion of the boundary. This corresponds to setting the
    groundwater table far from the well. Finally, $\{(x,z) \mid x =0\}$ is
    the inner wall of the well. The well is filled with water up to a
    fixed level $H_w$, and a hydrostatic Dirichlet condition for the
    pressure is applied along the portion of the boundary that is in
    contact with the body of water.  Above $H_w$, the seepage face
    boundary conditions apply. For the simulations presented here, we
    choose $H_w = 0.25m$. We remark that this simple setup and variations
    of it are common benchmarks for works on seepage problems
    \cite{oden1980theory,cooley1983seepage_face_algorithm,unconfinedseepagekazemzadeh2012,zheng2005new}.
    
    For the soil parametrisation, we make the choices $n = 2.06$, $\alpha
    = 1m^{-1}$, $K_S = 1{ms^{-1}}$.  This results in a soil that has the 
    characteristics
    of silt whose hydraulic conductivity has been scaled to have magnitude
    1. We note that in the stationary case when the right hand vanishes, this 
    has no effect on the pressure head.
    
    Figure \ref{well_1} shows an approximation to the the solution of the
    problem in this case, with the associated adjoint solution in Figure
    \ref{well_dual_1}. Notice the adjoint solution takes its largest
    values along the seepage face along which the quantity of interest is
    evaluated, with values increasing along streamlines that terminate
    there. This is to be expected as it demonstrates that the flow
    upstream of the seepage face has the greatest influence upon the
    quantity of interest.

    The simulation is initialised on a coarse mesh of 256 elements and
    uses the goal-based estimate as refinement criterion. A selection of
    meshes generated by the adaptive algorithm is given in figure
    \ref{ex_1_mesh1}--\ref{ex_1_mesh4}.  
    
    \begin{figure}[ht!]
        \centering 
        \begin{subfigure}{0.45\textwidth}
            \includegraphics[width=\linewidth]{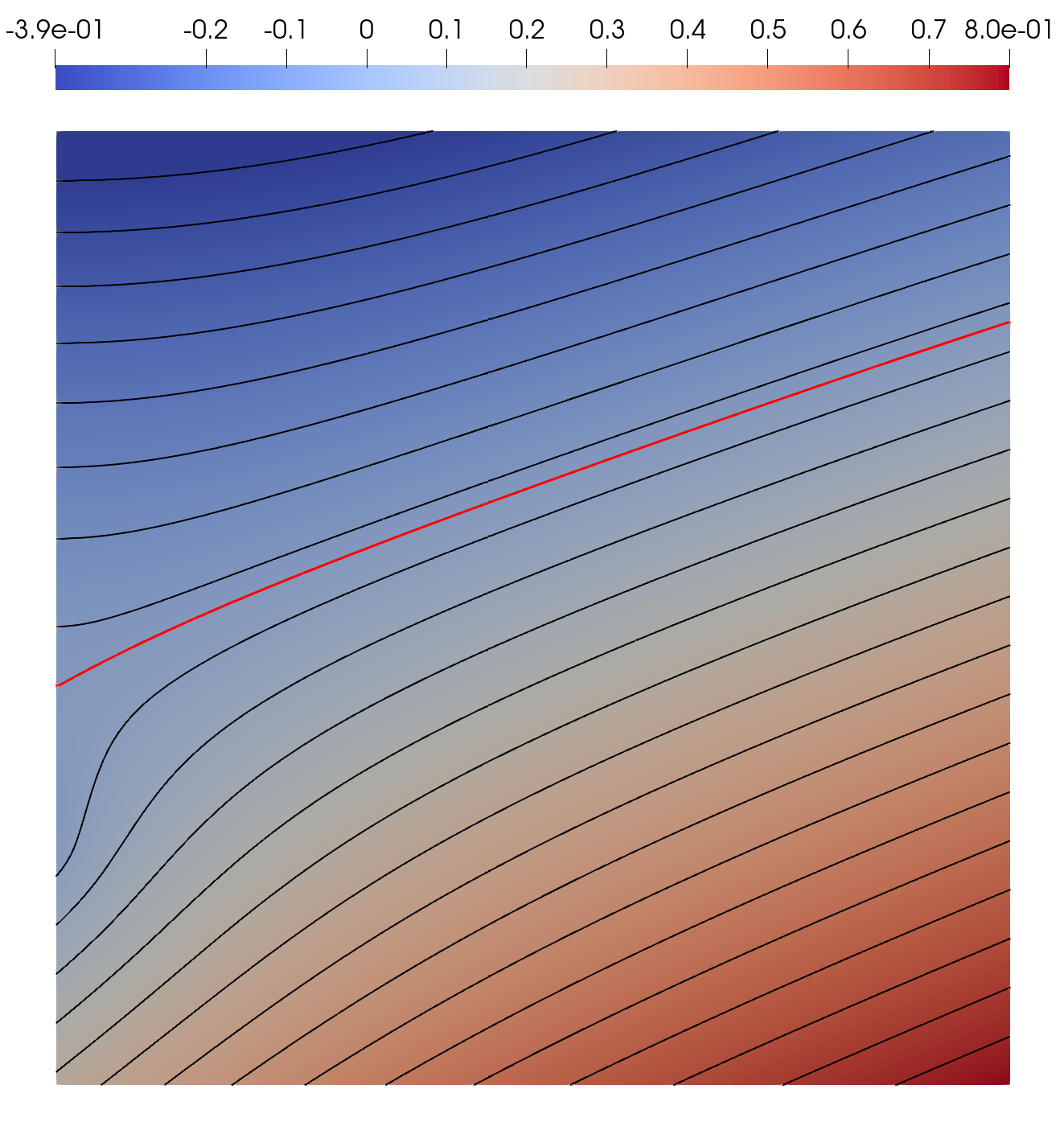}
            \caption{Contours of pressure. Level set $u_h = 0$ marked with red line.
                \label{well_1}
            }        
        \end{subfigure}\hfil 
        \begin{subfigure}{0.45\textwidth}
            \includegraphics[width=\linewidth]{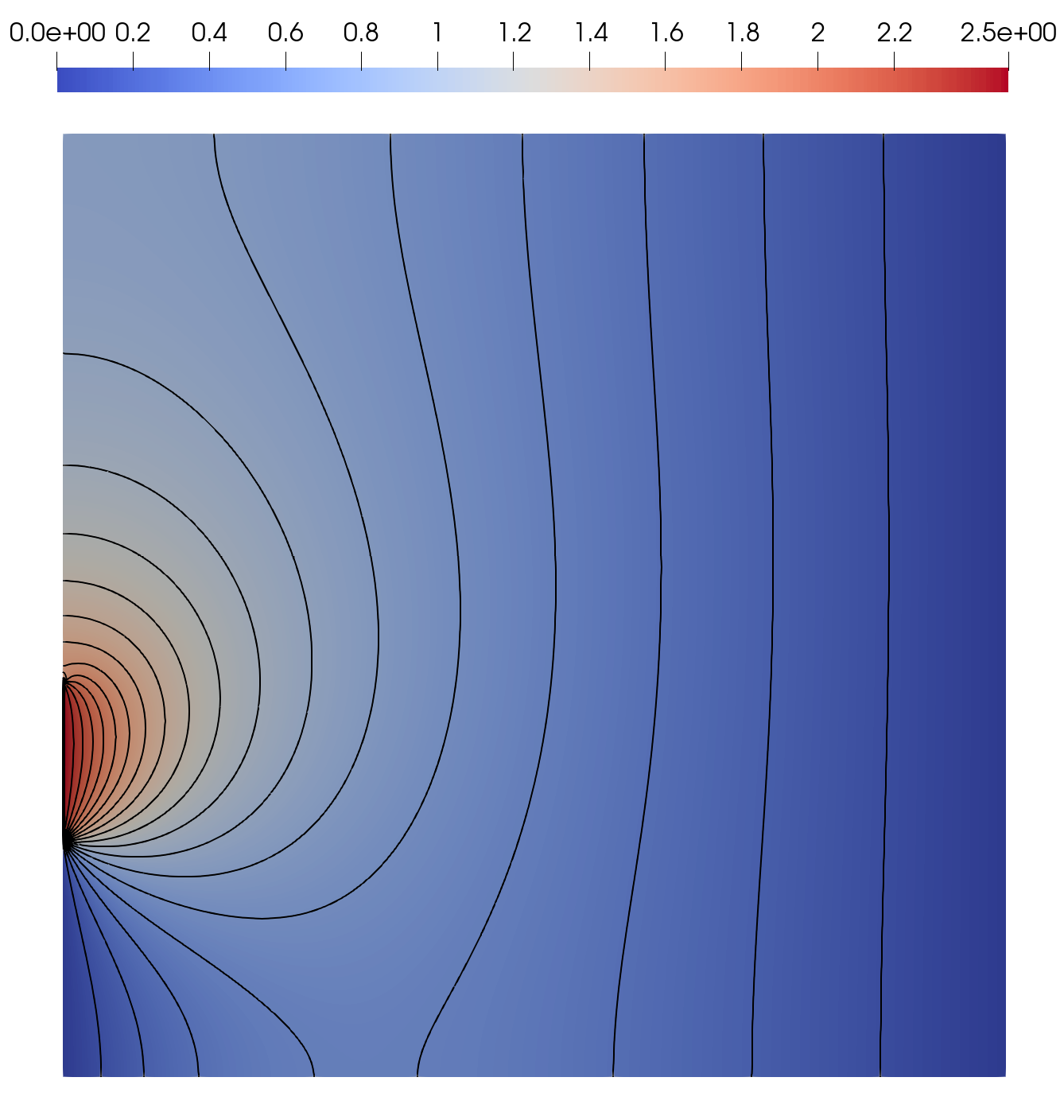}
            \caption{Contours of adjoint variable, $z_h$. Note that by definition 
            $z_h\geq 0$.
                \label{well_dual_1}
            }
        \end{subfigure}\hfil 
        \\
        \begin{subfigure}{0.24\textwidth}
            \includegraphics[width=\linewidth]{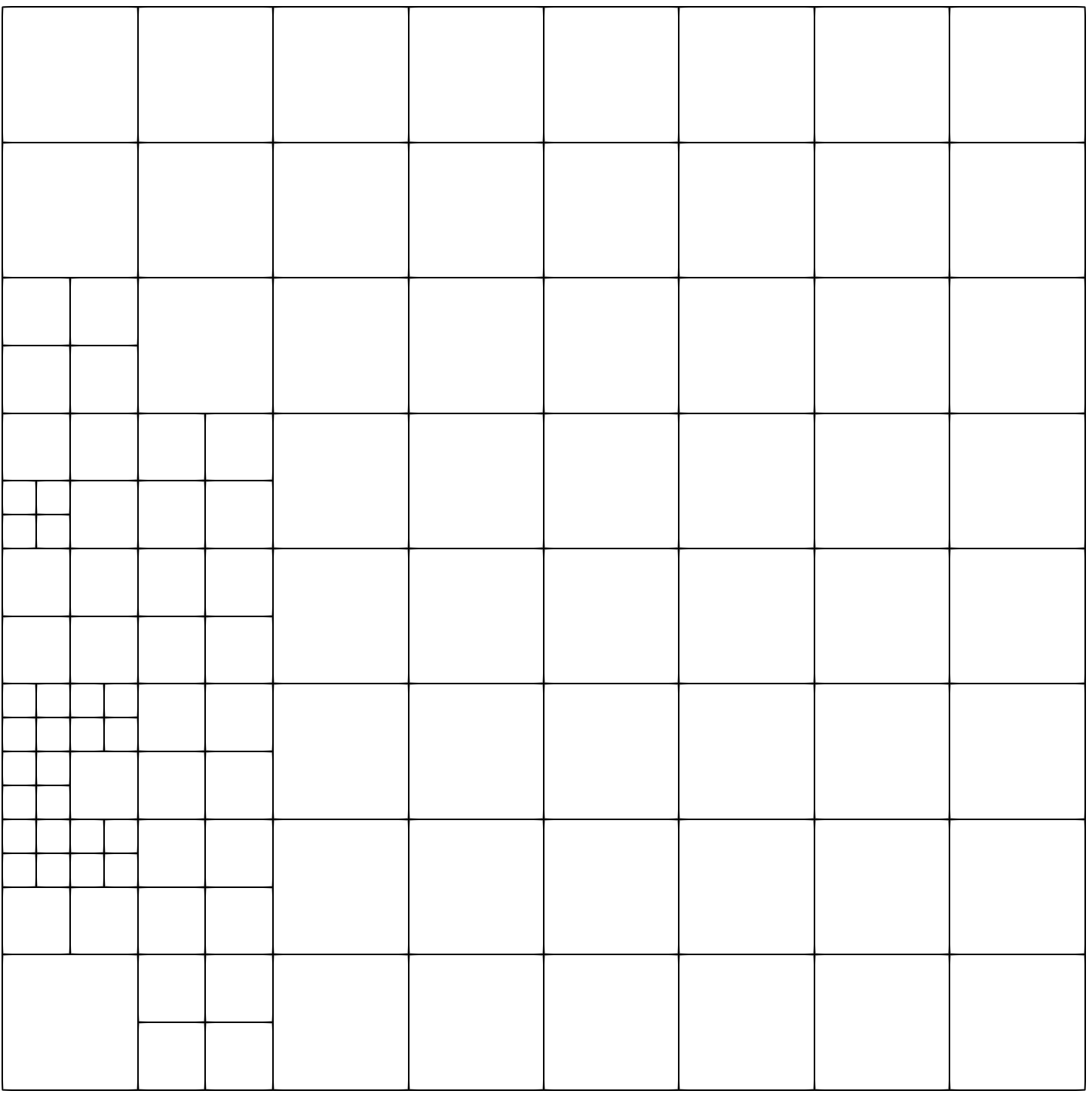}
            \caption{$\mathcal{T}^1$
                \label{ex_1_mesh1}}
        \end{subfigure}\hfil 
        \begin{subfigure}{0.24\textwidth}
            \includegraphics[width=\linewidth]{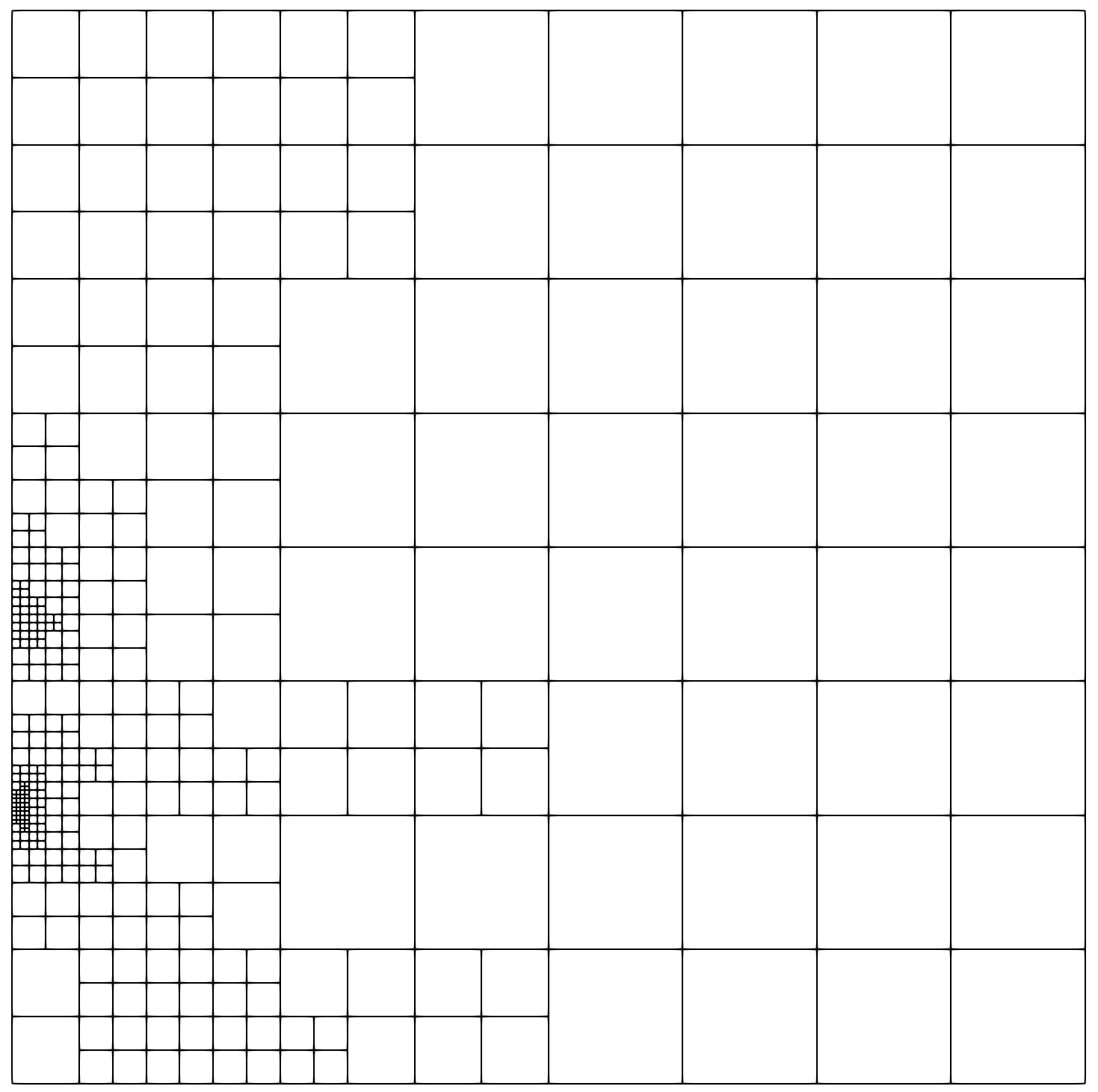}
            \caption{$\mathcal{T}^4$}
        \end{subfigure}\hfil 
        \begin{subfigure}{0.24\textwidth}
            \includegraphics[width=\linewidth]{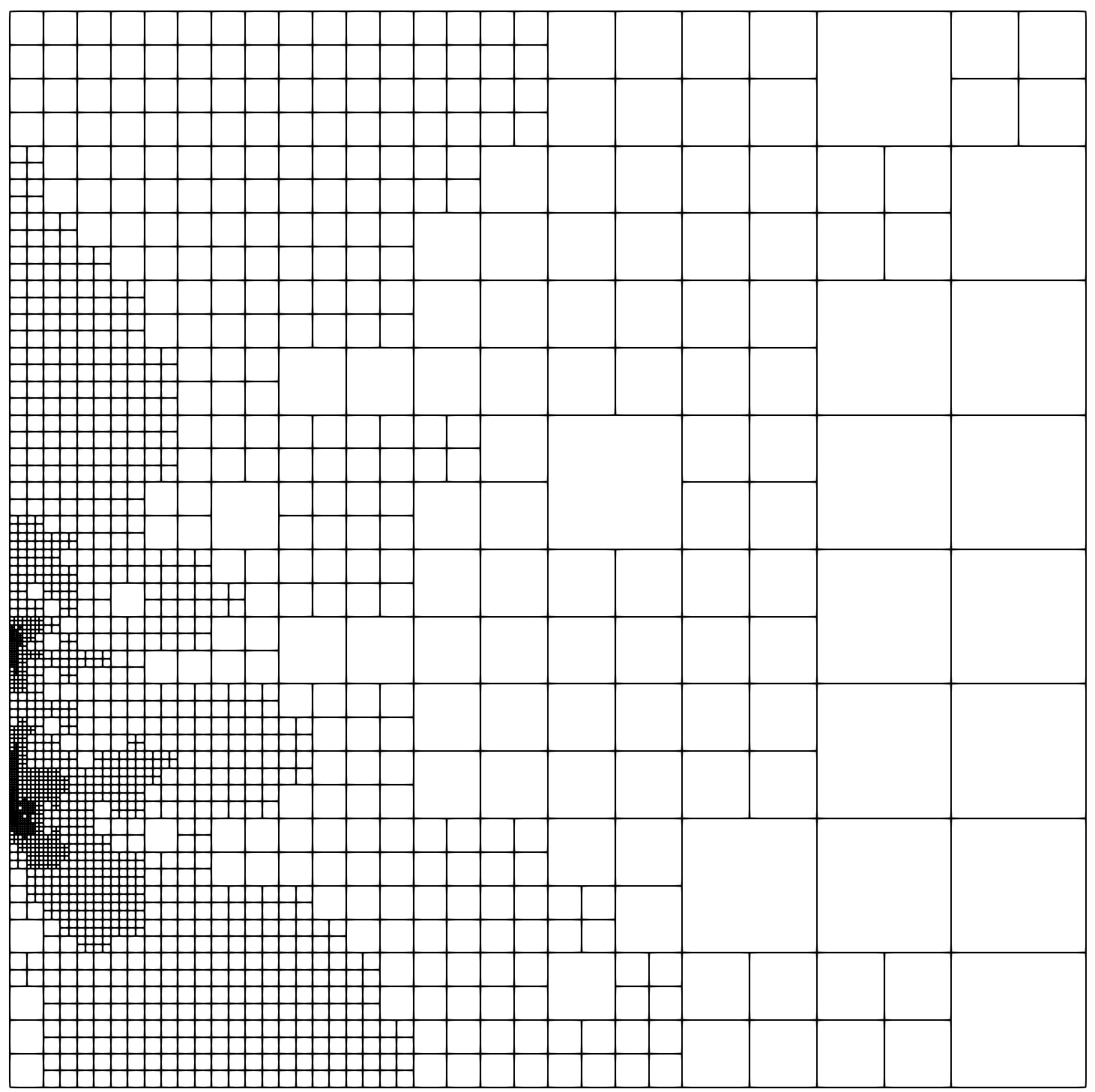}
            \caption{$\mathcal{T}^7$}
        \end{subfigure}\hfil 
        \begin{subfigure}{0.24\textwidth}
            \includegraphics[width=\linewidth]{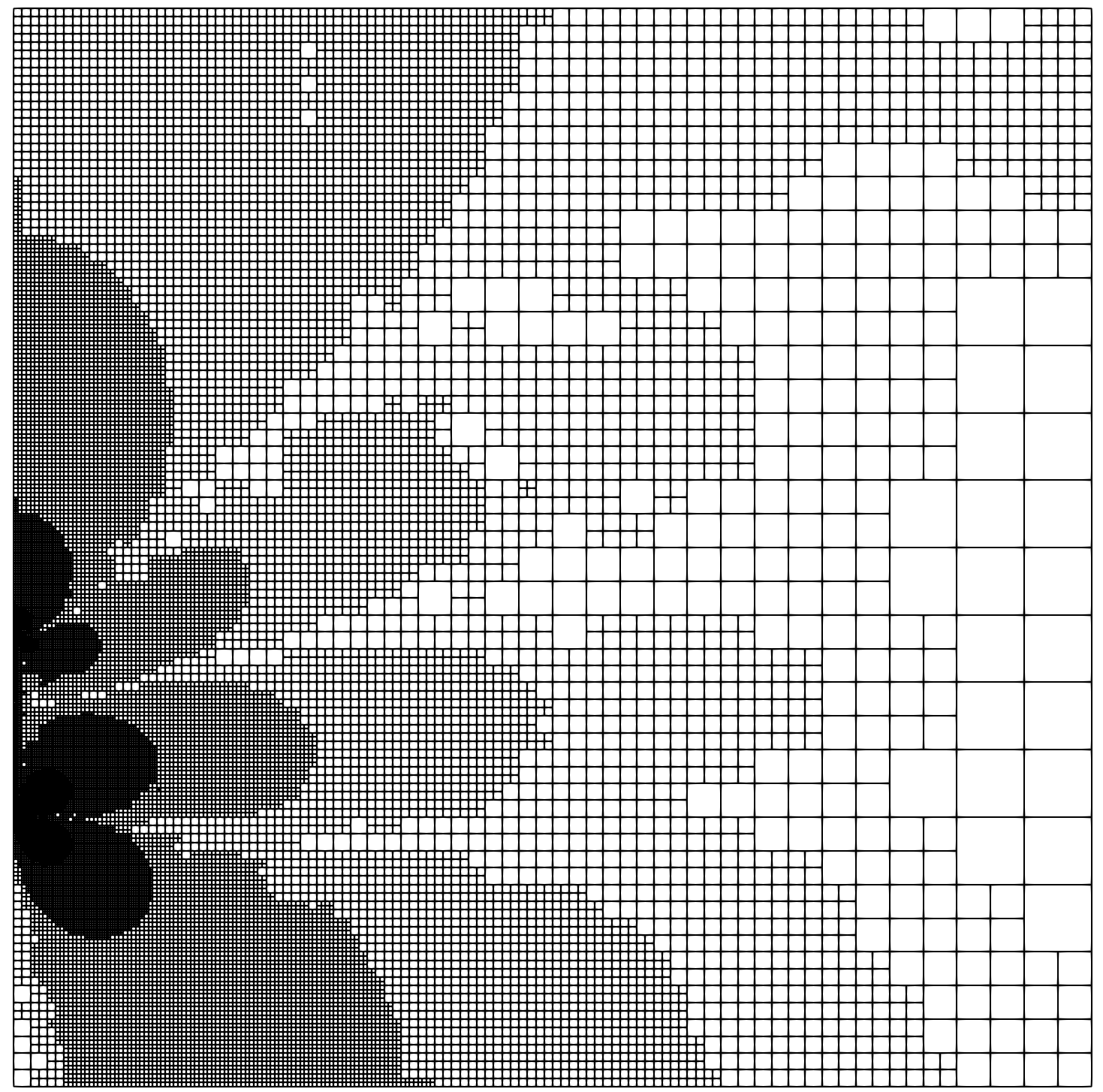}
            \caption{$\mathcal{T}^{11}$\label{ex_1_mesh4}}
        \end{subfigure}\hfil 
        \caption{Example 1, flow through a single layered, silty soil. We
            show the pressure, adjoint solution and a sample of adaptively
            generated meshes showing refinement upstream of the seepage
            face. The primal variable $u_h$ and the adjoint variable $z_h$
            are both represented on $\mathcal{T}^{11}$ which has
            approximately 66000 degrees of freedom.  }
    \end{figure}

    \subsection{Example 2: Sloping Unconfined Aquifer with Impeding 
    Layer}
    
    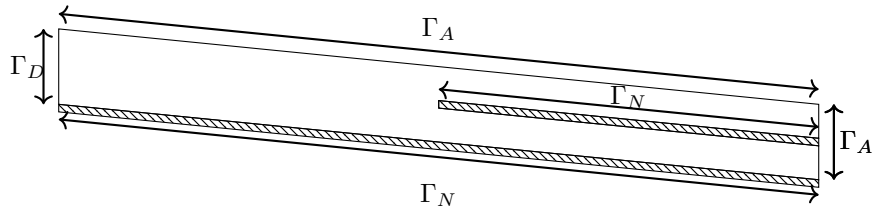
\begin{figure}[h!]
        \begin{center}
            \begin{tikzpicture}[scale=1.] 
                \draw (0,1) -- (0,2) -- (10,1) -- (10.,0.55) -- (5,1.05) -- (5,0.95) -- 
                (10,0.45) -- (10,0) -- cycle; 
                \draw [pattern=north west lines, pattern color=black] (0,1) -- (10, 0) 
                -- (10, -0.1) -- (0,0.9) -- cycle;
                \draw [pattern=north west lines, pattern color=black] (10.,0.55) -- 
                (5,1.05) -- (5,0.95) -- (10,0.45) -- cycle; 
                \draw[thick,<->] (0.,.8) -- (10.,-0.2);
                \node at (5, -.2) {$\Gamma_N$};
                \draw[thick,<->] (-0.2,1) -- (-0.2,2);
                \node at (-0.4,1.5) {$\Gamma_D$};         
                \draw[thick,<->] (0,2.2) -- (10,1.2);
                \node at (5,2.) {$\Gamma_A$};
                \draw[thick,<->] (5,1.2) -- (10,0.7);
                \node at (7.5,1.1) {$\Gamma_N$};
                \draw[thick,<->] (10.2,0) -- (10.2,1);
                \node at (10.5,.5) {$\Gamma_A$};
                \draw[thick,<->] (10.2,0) -- (10.2,1);
                \node at (10.5,.5) {$\Gamma_A$};
            \end{tikzpicture}
        \end{center}
        \caption{ The domain models a slope lying on a layer of bedrock
            with a downstream external boundary. The domain is a
            parallelogram with corners $(0,1)$, $(0,2)$, $(10, 1)$ and
            $(10,0)$ where all dimensions are in metres. The lower extent of
            the domain represents an impermeable boundary, as does a layer
            of rock parallel to the land surface towards the right hand side
            of the domain. This layer is 0.1m thick with corners $(5,0.95)$,
            $(5,1.05)$, $(10, 0.45)$ and $(10,0.55)$. The water table is
            fixed with a Dirichlet boundary condition on the left hand
            boundary of the domain.  }
        \label{fig:ex2}
    \end{figure}
    
    The second test case is taken from
    \cite{seepage_scudeler2017examination}. Its relevance was shown in
    \cite{rulon1985development} where the location of impeding layers was
    shown to have large effects on the saturation conditions of the
    soil. The domain setup is illustrated in Figure \ref{fig:ex2}. This
    configuration leads to water flowing down the slope due to gravity,
    and allows multiple seepage faces to form. We introduce a forcing
    term, representing an underground spring, above the layer to force
    extra seepage faces. It is defined by:
    \begin{equation}
        \label{forcing_term}
        f(x) = \begin{cases}
            10
            &\text{ if dist($x, (9,1.15)) < 0.2$}\\
            0 &\text{ otherwise.}
        \end{cases}
    \end{equation}
    
    We make the same choice of soil parameters as in example 1, that is $n = 
    2.06$, $\alpha
    = 1m^{-1}$, $K_S = 1{ms^{-1}}$. 
    
    The results of this are given in Figure \ref{ex_2_meshes}. As can be
    seen in Figure \ref{slope_2}, three disjoint seepage faces arise from
    this simulation, two on the right hand face, one above and one below
    the impermeable barrier, and another at the land surface. It should be
    noted that the seepage face at the land surface would generate surface
    run-off. This process is not taken into account by the model we use.
    
    The simulation is initialised on a coarse mesh of 4036 elements.
    An adaptive simulation using the dual-weighted estimate produced the
    meshes in figure \ref{ex_2_meshes}. The algorithm refines heavily
    around the source and all seepage faces, as well as resolving the
    corners around the impeding layer.
    
    \begin{figure}[h!]
        \begin{center}
            \begin{subfigure}{0.9\textwidth}
                \includegraphics[width=\linewidth]{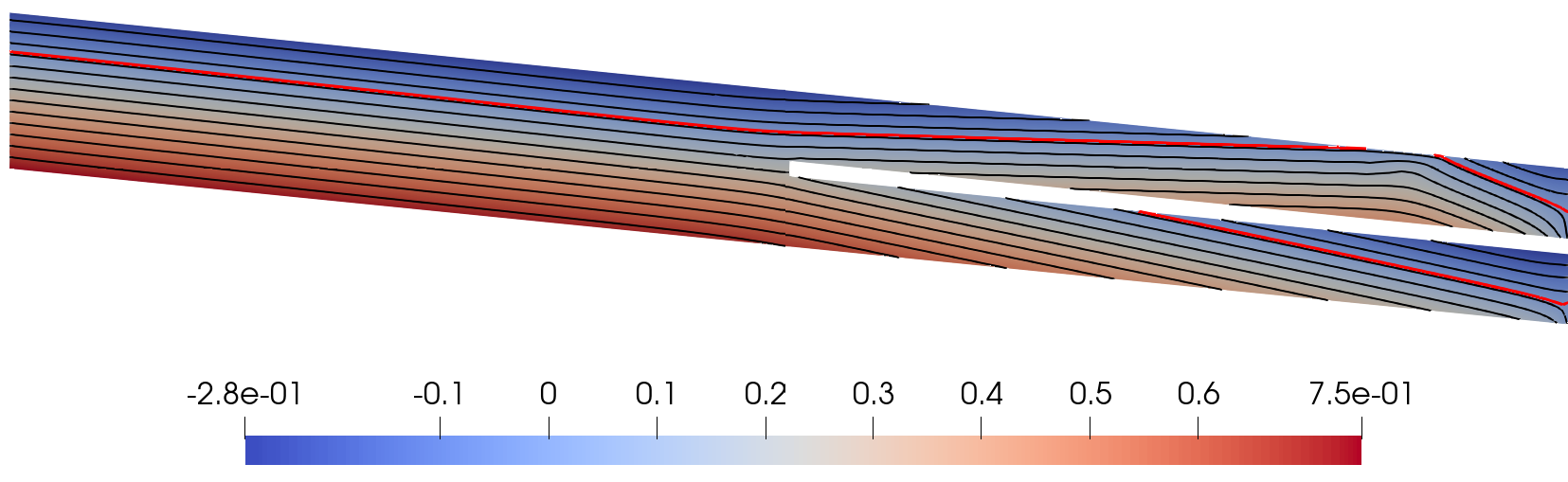}
                \caption{Simulation of hillside with water leak. Level set $u_h = 0$ 
                marked with red line.
                    \label{slope_2}
                }
            \end{subfigure}
            \begin{subfigure}{0.9\textwidth}
                \includegraphics[width=\linewidth]{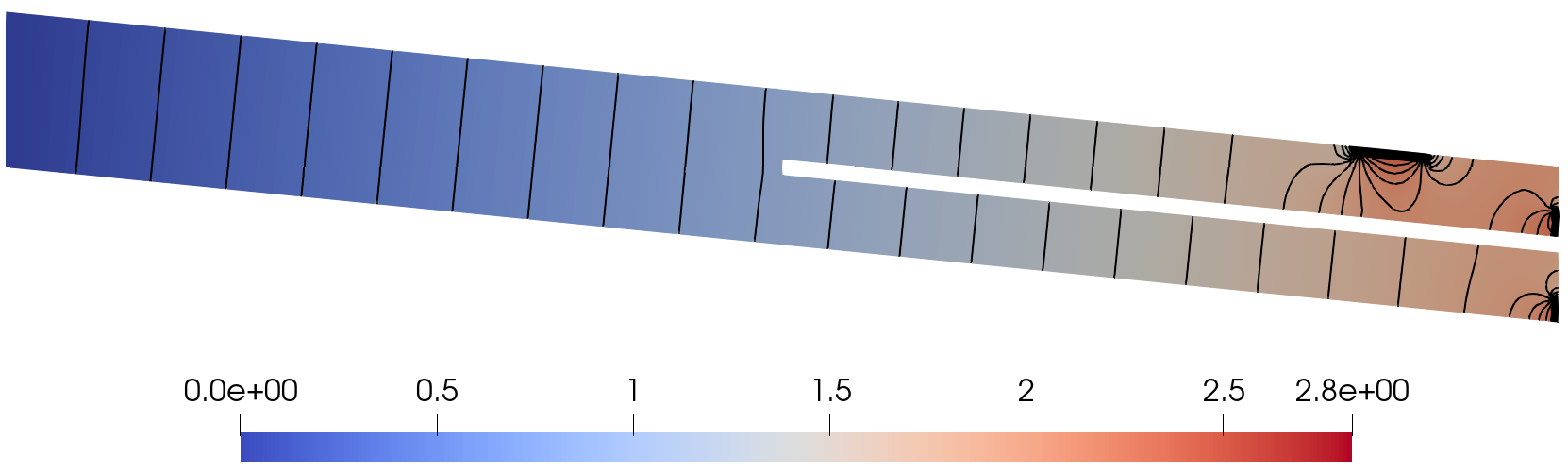}
                \caption{Contours of adjoint variable, $z_h$. Note the extreme 
                clustering of contours around the 
                    three seepage faces as well as high density around the source.
                    \label{slope_dual}
                }
            \end{subfigure}
            \\    \begin{subfigure}{0.45\textwidth}
                \includegraphics[width=\linewidth]{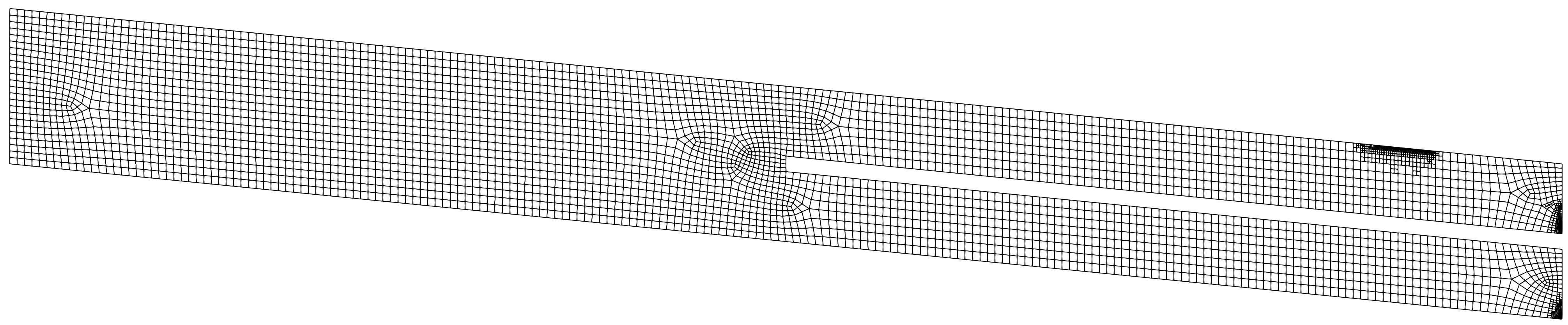}
                \caption{$\mathcal{T}^7$}
            \end{subfigure}\hfil 
            \begin{subfigure}{0.45\textwidth}
                \includegraphics[width=\linewidth]{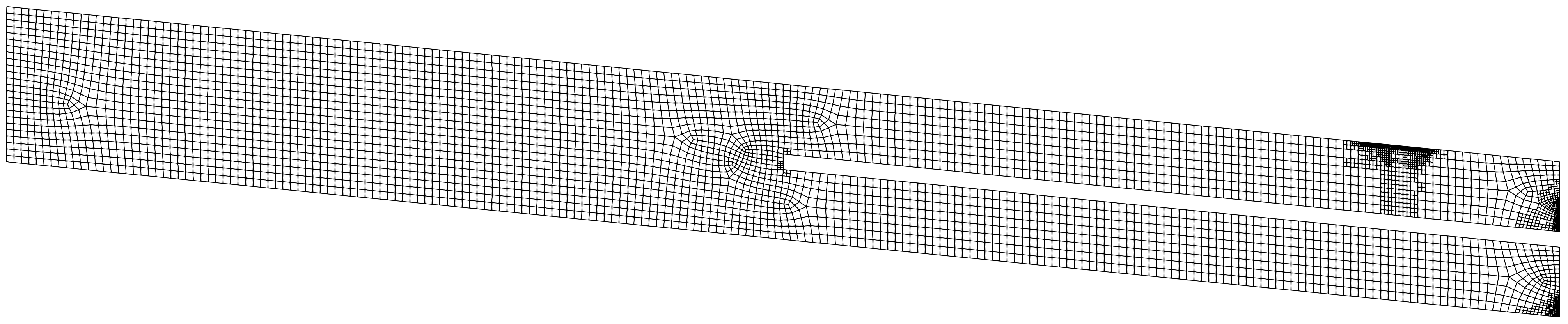}
                \caption{$\mathcal{T}^{10}$}
            \end{subfigure}\hfil 
            \begin{subfigure}{0.45\textwidth}
                \includegraphics[width=\linewidth]{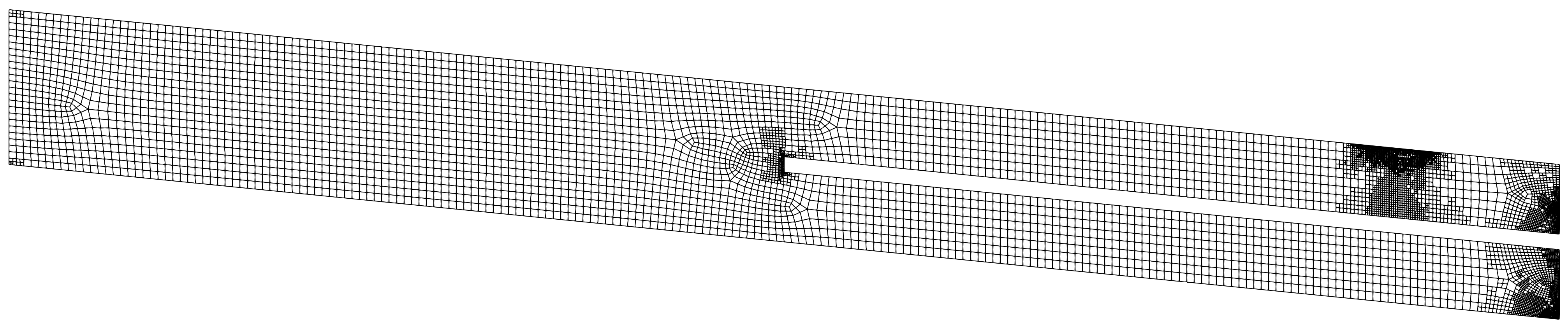}
                \caption{$\mathcal{T}^{14}$}
            \end{subfigure}\hfil 
            \begin{subfigure}{0.45\textwidth}
                \includegraphics[width=\linewidth]{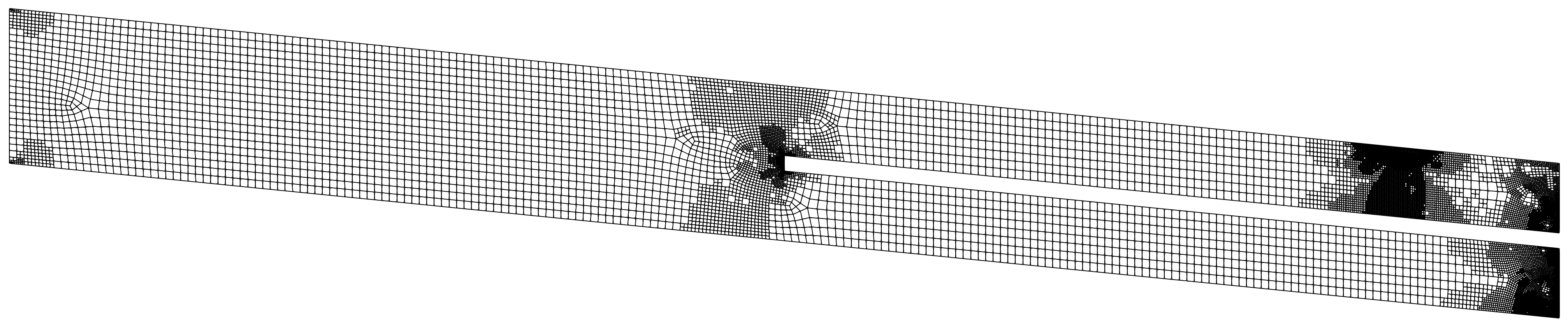}
                \caption{$\mathcal{T}^{17}$}
            \end{subfigure}\hfil 
            \caption{Example 2, flow through a sloped aquifer with impeding
                layer. We show the pressure, adjoint solution and a sample of
                adaptively refined meshes that capture multiple seepage faces as
                well as potential singularities in the pressure at the corner in
                the domain. The primal and adjoint variable are both represented
                on $\mathcal{T}^{17}$ which has approximately $7 \times 10^5$
                degrees of freedom.}
            \label{ex_2_meshes}
        \end{center}
    \end{figure}
    
    \subsection{Estimator Effectivity Summary}
    
    In Examples 1 and 2 above we compute a reference value for $J(u)$
    obtained from a simulation on a very fine grid. This was taken as the
    \lq true' value to perform analysis of the behaviour of the
    estimate. In Figures \ref{fig:sharpness_1}--\ref{fig:sharpness_2}, we
    see that as the simulation progresses the effectivity of the estimate,
    defined as the ratio of the error to the estimate, becomes very close
    to 1.
    
    \begin{figure}[!h]
        \centering 
        \begin{subfigure}{0.45\textwidth}
            \includegraphics[width=\linewidth]{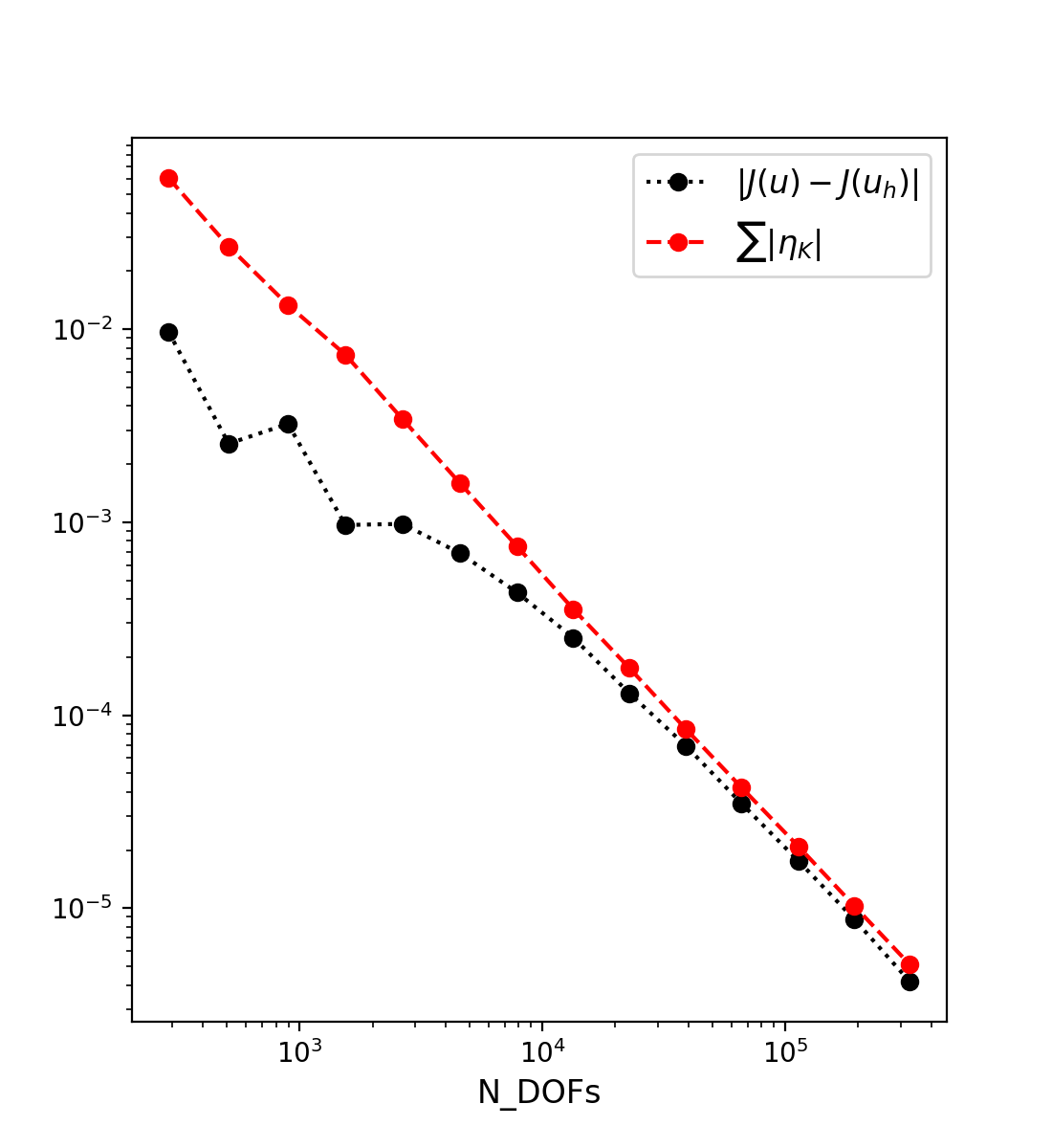}
            \caption{Example 1.}
            \label{fig:sharpness_1}
        \end{subfigure}\hfil 
        \begin{subfigure}{0.45\textwidth}
            \includegraphics[width=\linewidth]{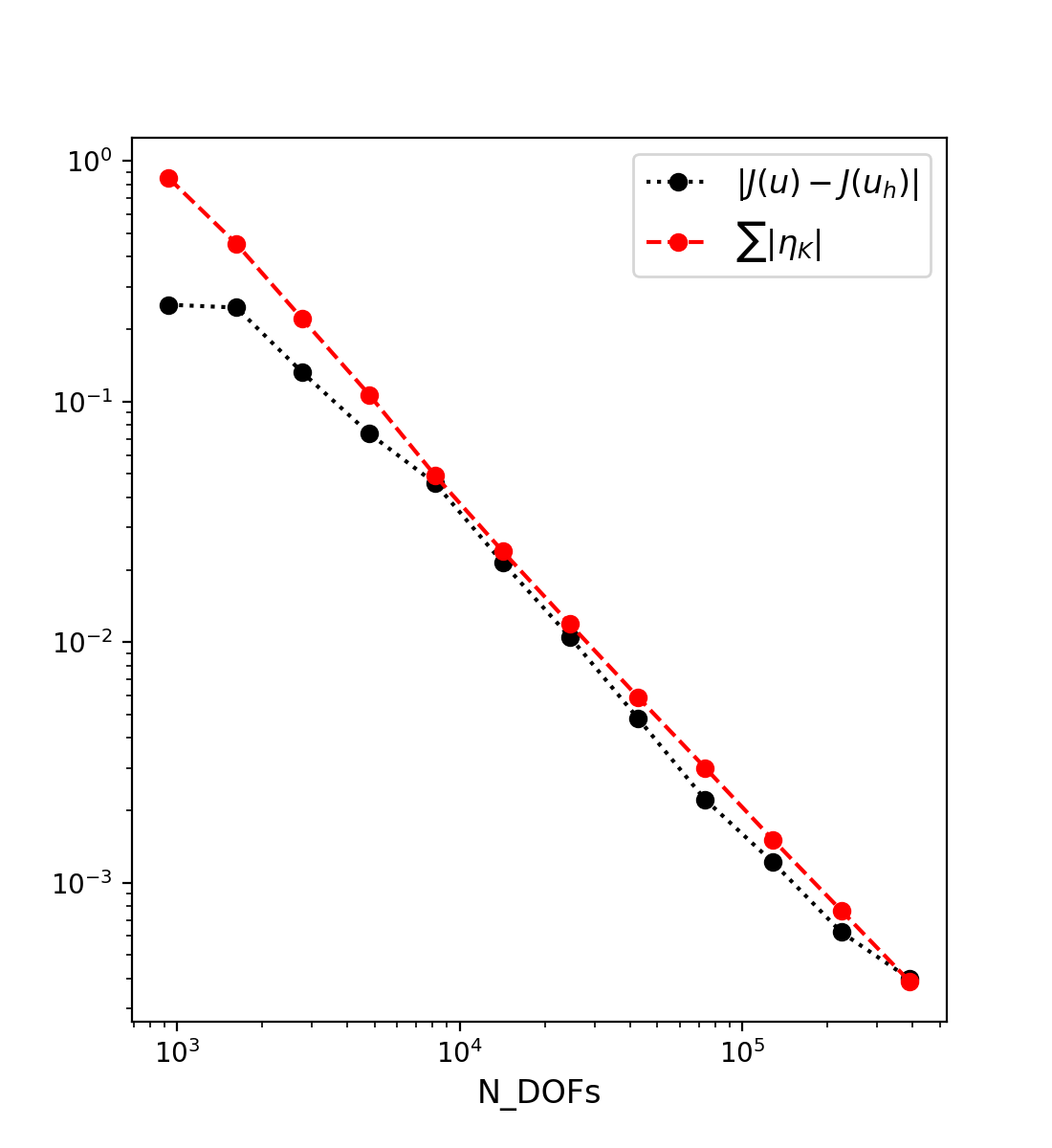}
            \caption{Example 2.}
            \label{fig:sharpness_2}
        \end{subfigure}
        \caption{
            Sharpness of error estimates during adaptive mesh refinement.
            Notice that the dual-weighted estimate significantly
            under-estimates the error for the first few refinement cycles
            but as the simulation progresses the effectivity moves closer to
            one. This is a well known feature of this class of algorithm
            further described in \cite{nochetto2008safeguarded}}
        \label{sharpness}
    \end{figure}
    
    \subsection{Adaptive vs Uniform Comparison}
    
    To illustrate the gains obtained through adaptive refinement, we make
    a comparison between the a uniformly refined simulation and the
    adaptive one. In each case uniform meshes perform extremely poorly
    with small and unpredictable reductions in error where the adaptive
    scheme produces fast and monotonic error reduction on all but the
    coarsest meshes. For comparison, two lines illustrating different
    rates are included in figure \ref{fig:orders_1} illustrating that
    convergence of $J(u_h)$ is suboptimal for uniform meshes, and that in
    terms of degrees of freedom, this optimality can be restored using the
    goal-based estimate.
    
    \begin{figure}[!h]
        \centering 
        \begin{subfigure}{0.45\textwidth}
            \includegraphics[width=\linewidth]{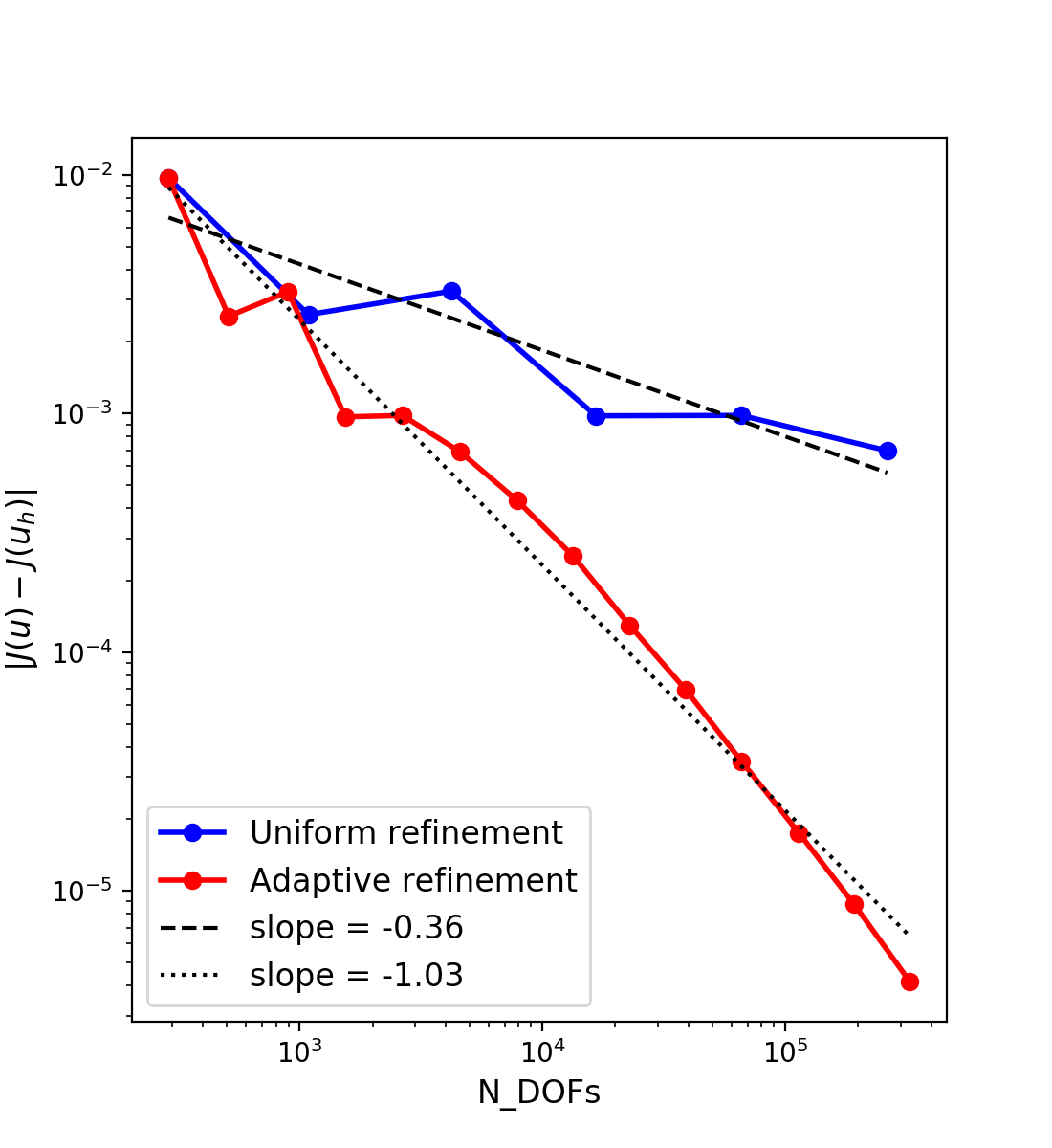}
            \caption{Example 1.}
            \label{fig:orders_1}
        \end{subfigure}\hfil 
        \begin{subfigure}{0.45\textwidth}
            \includegraphics[width=\linewidth]{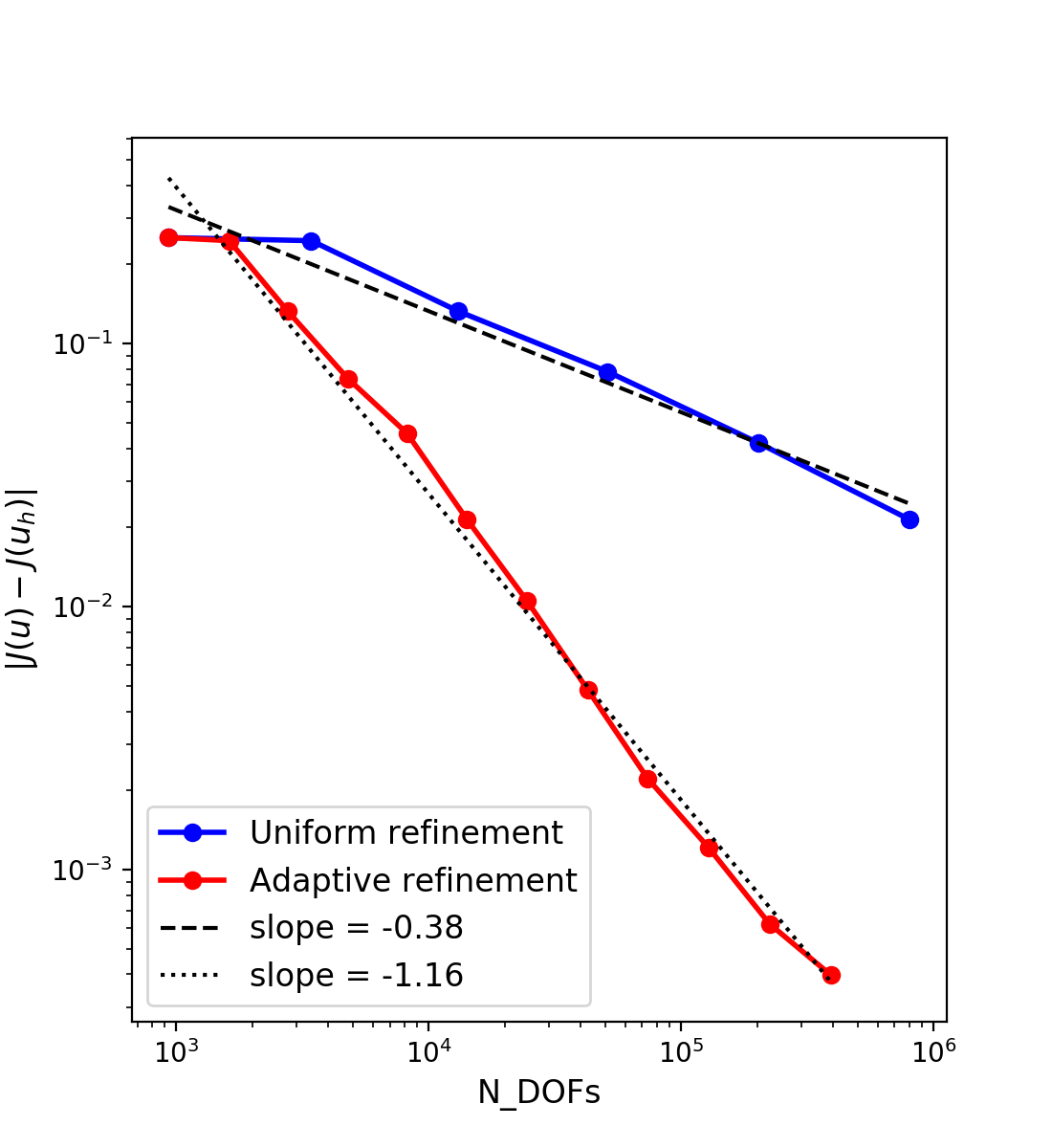}
            \caption{Example 2.}
            \label{fig:orders_2}
        \end{subfigure}
        \caption{Comparison of orders of convergence in terms of number of
            degrees of freedom ($N_{DOFS}$) on uniform and adaptive
            grids. Notice the rate of error reduction is considerably slower
            for uniform simulations in all cases.}
        \label{orders}
    \end{figure}
    
    \clearpage
    
    \section{Case Studies with Layered Inhomogeneities}
    \label{sec:case}
    
    We present results making use of borehole data provided by CPRM
    (Brazilian Geological Survey) by the Siagas
    system \footnote{\url{http://siagasweb.cprm.gov.br/layout/}}. The
    wells are used to supply water to two different cities in S\~ao Paulo
    State, Brazil, one in Ibir\'a and the other in Porto Ferrreira.  Both
    cities are located over the Paran\'a Sedimentary Basin, but in places
    with different shallow geology. There are two different problem setups
    that we consider. In both cases the domain is a vertical section
    illustrated in Figure \ref{ex3:schematic}. We assumed the soil is in
    homogeneous layers, where there is no variation in the physical
    properties in the horizontal direction. The soil parameters used for
    the simulation are given in Table \ref{parameters}.  The water table
    height far from the well is known and applied as a Dirichlet boundary
    condition for the pressure on the right hand lateral boundary. In both
    cases, the height of the water in the well gives the left lateral
    boundary condition, and water is continually pumped out of the well in
    such a way that the water height remains constant.

    We work in cylindrical coordinates with the $(r, \phi, z)$ with the
    $z$-axis aligned with the centre of the well. The aim is to calculate
    the total flux into the well. We therefore use the functional $J_2$
    to account for flux of water over the inner boundary below the water level, 
    defined as follows.
    
    \begin{equation} \label{functional_augmented}
        J_2(u) 
        :=
        2 \pi r_0 \int_{r = r_0}\mathbf{q(u)}\cdot \mathbf{n} \diff z,
    \end{equation}
    where $r_0$ denotes the radius of the well, that is, we integrate over the 
    entire inner wall of the well, 
    above 
    and below the water.

    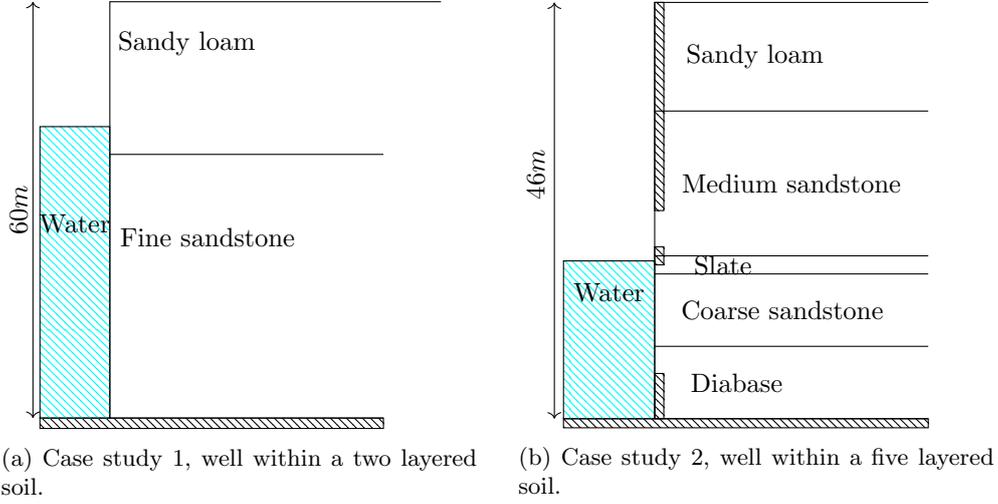
\begin{figure}[h!]
        \begin{center}
            \begin{subfigure}{0.45\textwidth}
                \begin{tikzpicture}[scale=0.12*46/60]
                    \draw [<->] (-11,-60) -- (-11,0);
                    \node[text=black,rotate=90] at (-13,-30) {$60m$};
                    \draw (30*60/38,0) -- (0,0) -- (0,-60) -- (30*60/46,-60); 
                    \draw [pattern=north west lines, pattern color=black] (-10,-60) -- 
                    (30*60/46, -60) -- (30*60/46, -61.5) -- (-10,-61.5) -- cycle;
                    \draw [pattern=north west lines, pattern color=cyan] (-10,-60) -- 
                    (-10, -18) -- (0, -18) -- (0,-60) -- cycle;
                    \node[text=black] at (-5,-32) {Water};
                    \draw (0,-22) -- (30*60/46,-22);
                    \node[text=black] at (11,-6) {Sandy loam};
                    \node[text=black] at (14,-34) {Fine sandstone};
                \end{tikzpicture}
                \caption{Case study 1, well within a two layered soil.}
            \end{subfigure}
            \hfil
            \begin{subfigure}{0.45\textwidth}
                \begin{tikzpicture}[scale=0.12] 
                    \draw [<->] (-11,-46) -- (-11,0);
                    \node[text=black,rotate=90] at (-13,-19) {$46m$};
                    \draw (30,0) -- (0,0) -- (0,-46) -- (30,-46); 
                    \draw [pattern=north west lines, pattern color=black] (-10,-46) -- 
                    (30, -46) -- (30, -47) -- 
                    (-10,-47) -- cycle;
                    \draw [pattern=north west lines, pattern color=cyan] (-10,-46) -- 
                    (-10, -28.56) -- (0, -28.56) -- 
                    (0,-46) -- cycle;
                    \draw [pattern=north west lines, pattern color=black] (0,-46) -- (0, 
                    -41) -- (1, -41) -- 
                    (1,-46) -- cycle;
                    \draw [pattern=north west lines, pattern color=black] (0,-29) -- (0, 
                    -27) -- (1, -27) -- 
                    (1,-29) -- cycle;
                    \draw [pattern=north west lines, pattern color=black] (0,-23) -- (0, 
                    0) -- (1, 0) -- 
                    (1,-23) -- cycle;
                    \node[text=black] at (-5,-32) {Water};
                    \draw (0,-12) -- (30,-12);
                    \node[text=black] at (11,-6) {Sandy loam};
                    \draw (0,-28) -- (30,-28);
                    \node[text=black] at (15,-20) {Medium sandstone};
                    \draw (0,-30) -- (30,-30);
                    \node[text=black] at (7.5,-29) {Slate};
                    \node[text=black] at (14,-34) {Coarse sandstone};
                    \draw (0,-38) -- (30,-38);
                    \node[text=black] at (9,-42) {Diabase};
                \end{tikzpicture}
                \caption{Case study 2, well within a five layered soil.}
                
            \end{subfigure}
        \end{center}
        \caption{
            \label{ex3:schematic}
            Geometric setup of the industrial case study problems. Black
            shading represents an impermeable boundary. In case study two, the
            gaps between impermeable regions on the inner wall of the well are
            the filter locations. The far field boundary conditions are
            analogous to those given in Figure \ref{schematic} and water is
            continually pumped out to maintain constant water height.}
    \end{figure}
    
    \subsection{Case Study 1 - 2 layered well in Ibir\'a (CPRM reference 
    3500023601)}
    
    For these case studies, all lengths are given in metres. In the first
    case we set $\Omega = \{(r,\phi,z) \mid 0.0762 \leq r \leq 50, \, 0
    \leq z \leq 60 \}$. The medium consists of sandy loam for $38 \leq z
    \leq 60$ and fine sandstone for $0 \leq z \leq 38$. We refer to Table
    \ref{parameters} for details of the parametrisations of these
    soils. Again, the base of the well is assumed to consist of impervious
    rock, and a no-flow boundary condition is enforced. There is assumed
    to be no water flow at the land surface. The water table has been
    measured in the vicinity of the well to be 49.8$m$, so we set a
    hydrostatic boundary condition at $r = 50$ to represent the far field
    conditions around the well. The height of water in the well is
    42.7$m$. The initial mesh is aligned with the layers in the soil. The
    solution, together with a selection of adaptive meshes are given in
    Figure \ref{real_well_1}. The computed flux as a function of degrees
    of freedom is given in Figure \ref{fig:well_1} showing that the
    mathematical model is in good agreement with the experimental data.

    \setlength{\tabcolsep}{12pt}
    \begin{table}
        \centering
        \caption{Case study soil parameters. Parameters used in the van
            Genuchten-Mualem model for hydraulic conductivity in each of the
            several types of soil and rock. Note the differences of several
            orders of magnitude in the parameters $K_S^i$, causing strong
            discontinuities in the coefficient $k$.}
        \label{parameters}
        \begin{tabular}[t]{|c|c|c|c|}
            \hline
            Layer &$K_S$ ($m s^{-1}$)& $n$  & $\alpha$ ($m^{-1}$) \\
            \hline
            sandy loam         &  5E-6   &  1.65 & 0.66  \\
            med. sandstone & 9E-6   &  1.36 & 0.012 \\
            slate                      & 5.0E-9  & 6.75 & 0.98 \\
            fine sandstone    & 1.15E-6  & 1.361 & 0.012 \\
            diabase                & 2E-5      & 1.523 & 1.066 \\
            \hline
        \end{tabular}
    \end{table}

    \begin{figure}[!h]
        \centering 
        \begin{subfigure}{0.4\textwidth}
            \includegraphics[width=\linewidth]{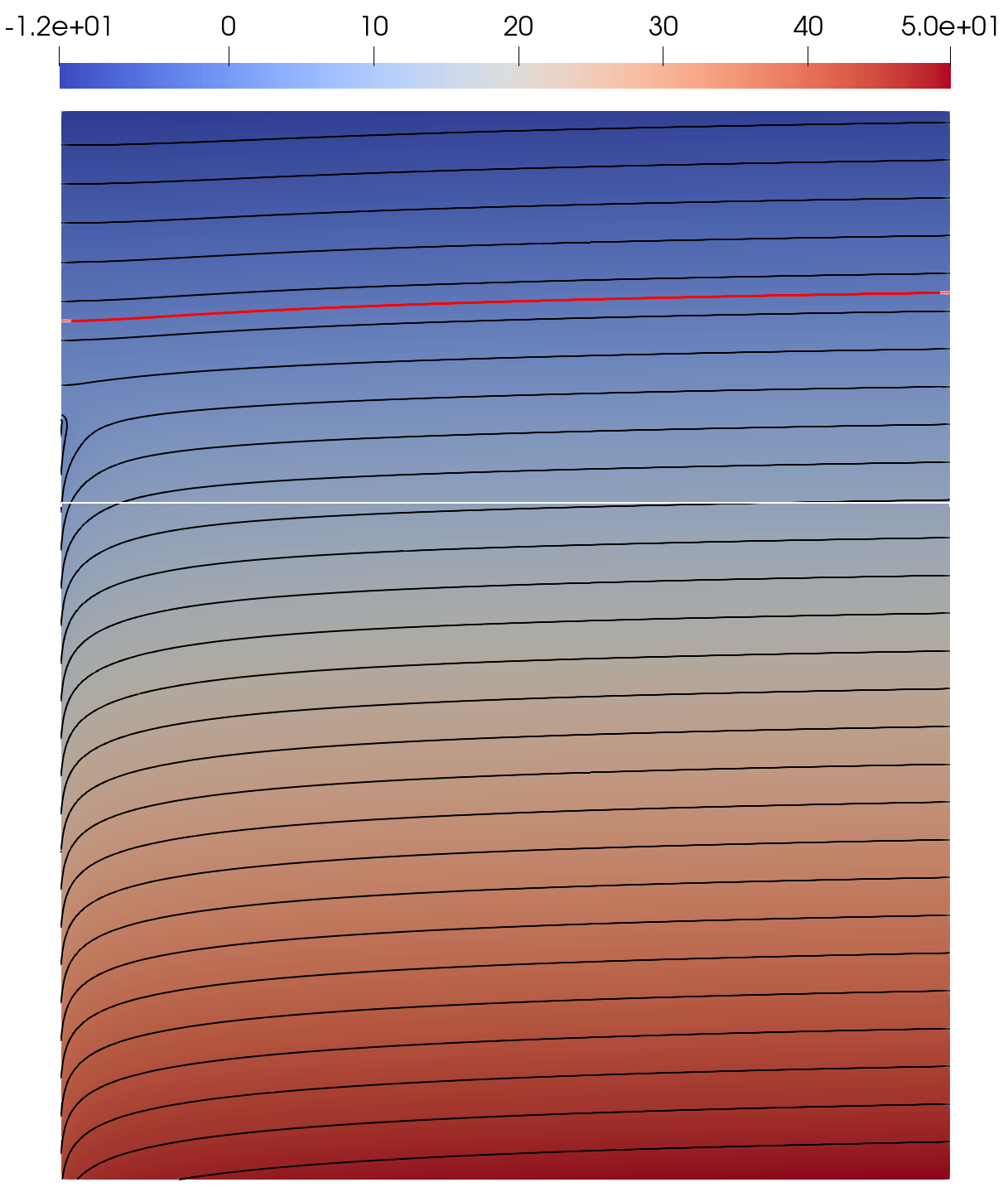}
            \caption{Contours of pressure.}
            \label{well_number_1}
        \end{subfigure}\hfil 
        \begin{subfigure}{0.4\textwidth}
            \includegraphics[width=\linewidth]{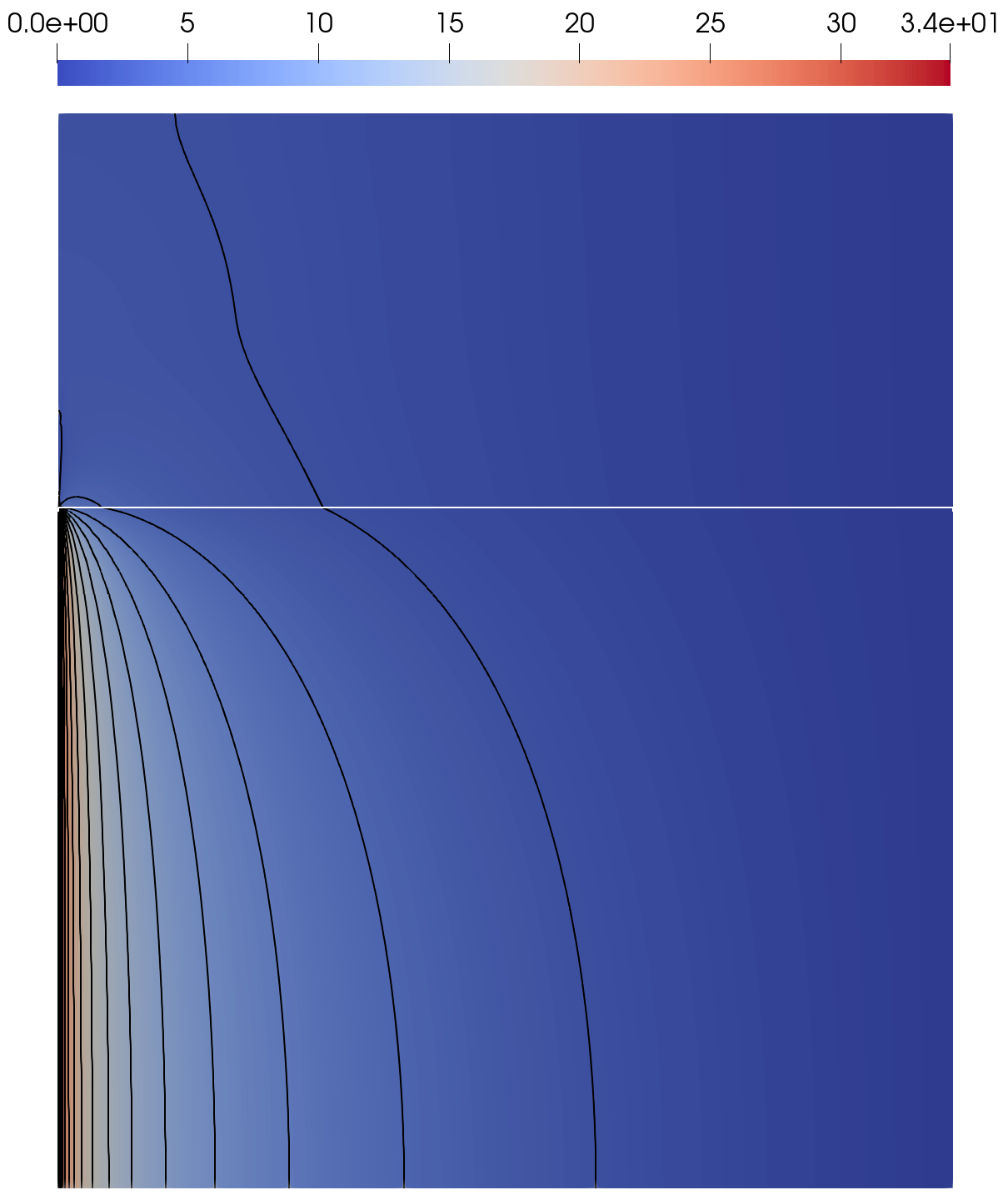}
            \caption{Contours of adjoint pressure, $z_h$.}
            \label{well_dual_number_1}
        \end{subfigure}
        \\ 
        \begin{subfigure}{0.2\textwidth}
            \includegraphics[width=\linewidth]{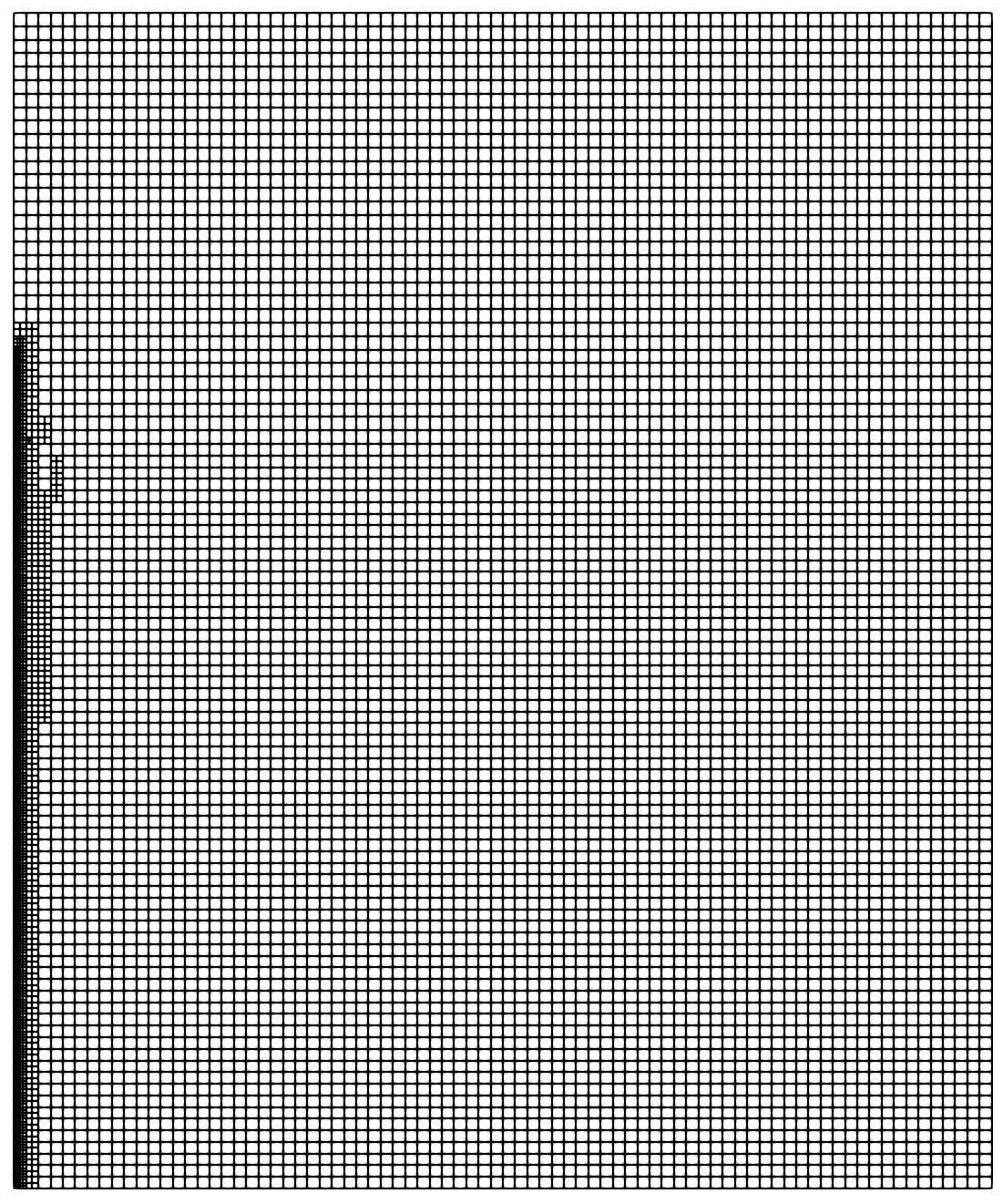}
            \caption{$\mathcal{T}^{15}$}
        \end{subfigure}\hfil 
        \begin{subfigure}{0.2\textwidth}
            \includegraphics[width=\linewidth]{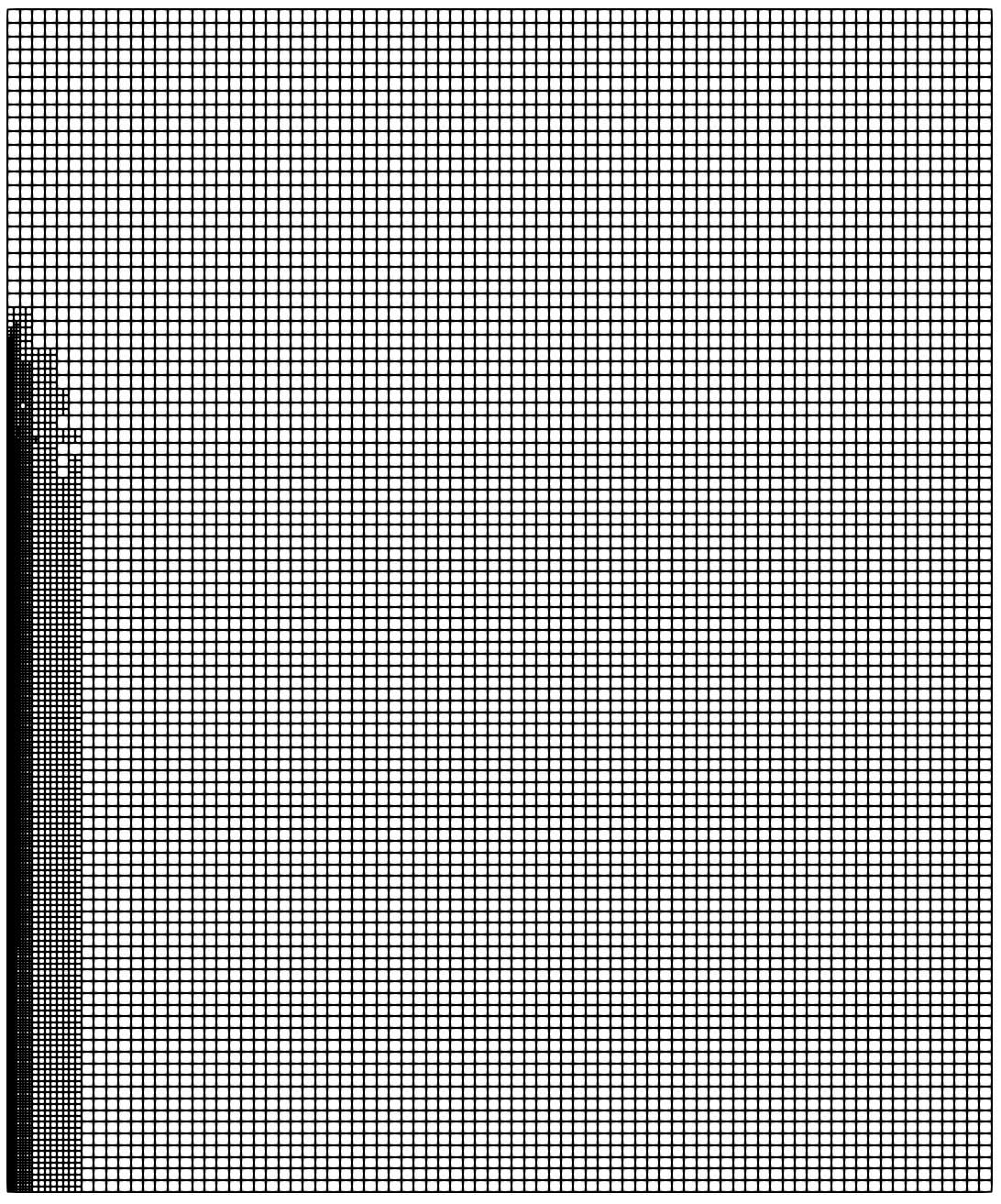}
            \caption{$\mathcal{T}^{20}$}
        \end{subfigure}\hfil 
        \begin{subfigure}{0.2\textwidth}
            \includegraphics[width=\linewidth]{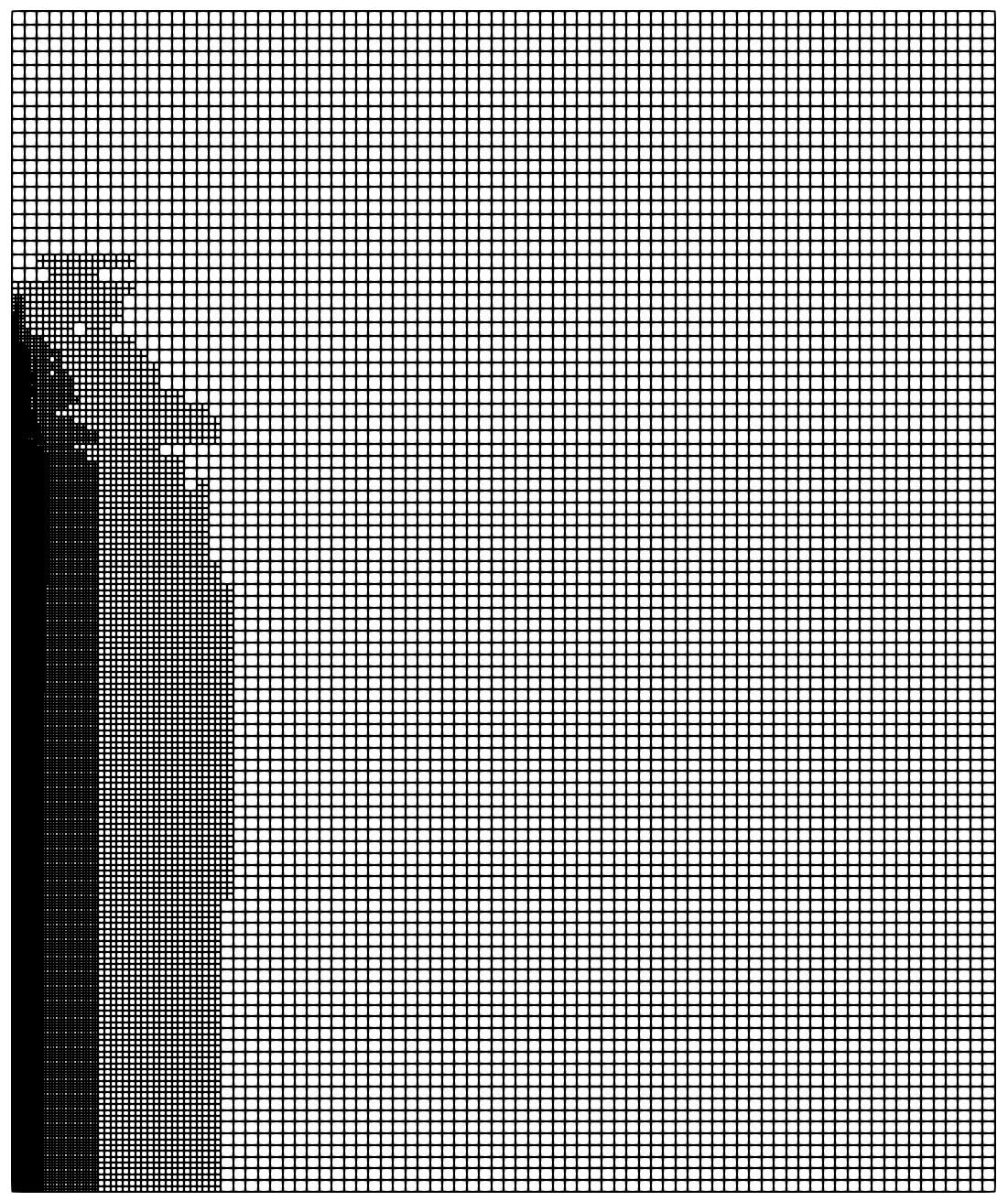}
            \caption{$\mathcal{T}^{30}$}
        \end{subfigure}\hfil 
        \begin{subfigure}{0.2\textwidth}
            \includegraphics[width=\linewidth]{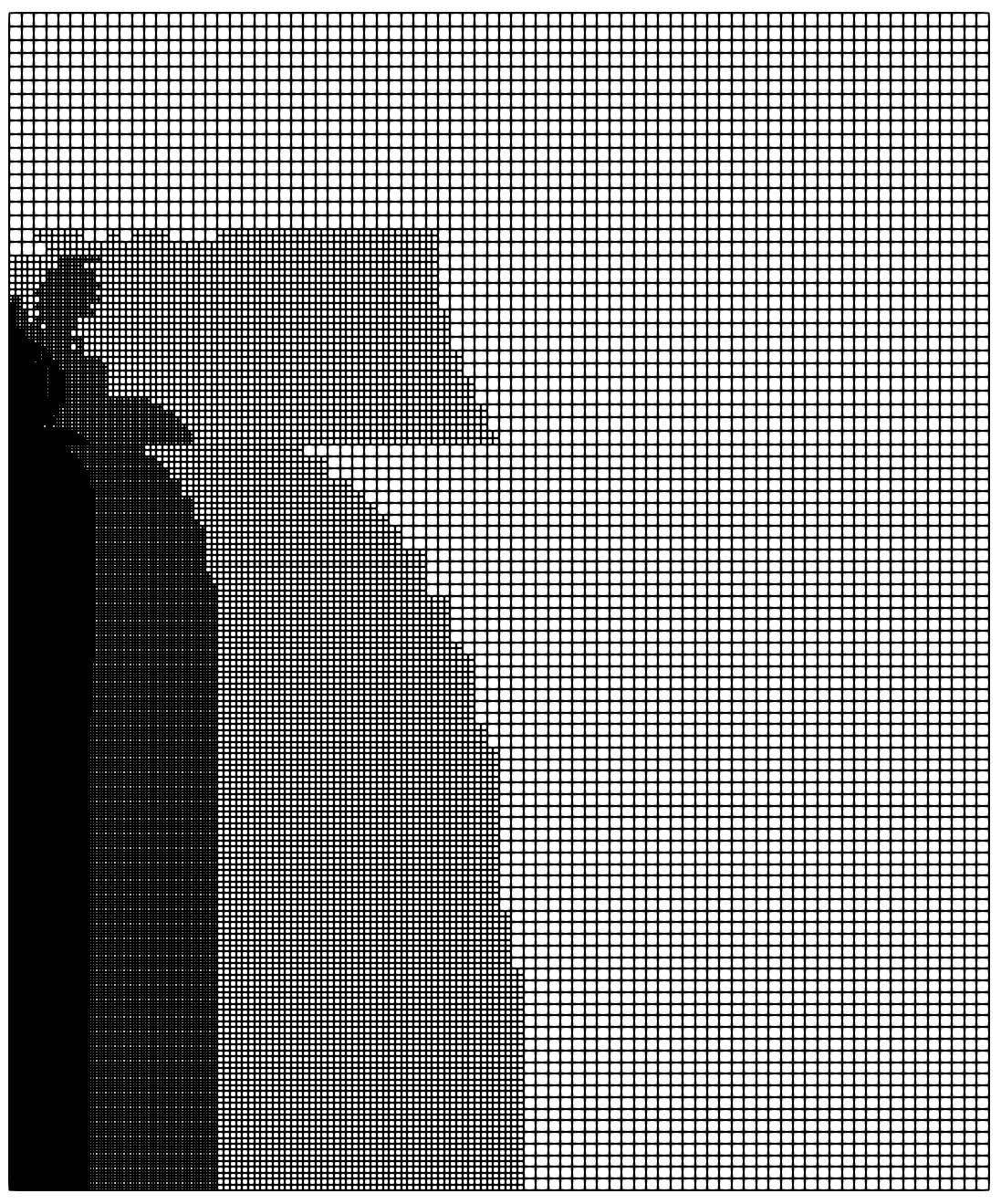}
            \caption{$\mathcal{T}^{35}$}
        \end{subfigure}\hfil 
        \caption{Case study 1, flow through a two layered soil. We show
            the pressure, the adjoint solution and a sample of adaptively
            generated meshes. The boundary between the soil layers is 
            marked with a white line. Both solutions are represented on 
            $\mathcal T^{35}$ which has approximately 1.5 million degrees
            of freedom.
        }
        \label{real_well_1}
    \end{figure}
    
    \subsection{Case study 2 - 5 layered well in Porto Ferreira (CPRM 
    reference 3500009747)}
    The second case study is a challenging setup with five layers of
    highly varying hydraulic properties, as well as complex boundary
    conditions due to the fact that in this case the inner wall of the
    well is impermeable apart from two filters to allow water to flow into
    the well. One is below and one above the water, meaning that the
    former allows flow into the subsurface and the other allows flow
    out. Along the inner wall, filters cover the part of the wall with $5
    \leq z \leq 17$ and $19 \leq z \leq 23$. The water level in the well
    is set at 17.44, with the other boundary conditions as in case study
    1, with the water table at the far boundary set at 33.9. Once again we
    assume a radially symmetric solution. The domain is given by $\Omega =
    \{(r,\phi,z) \mid 0.1585 \leq r \leq 50, \, 0 \leq z \leq 46 \}$. The
    medium consists of five layers. In order, with the top layer first,
    the layers consist of sandy loam, medium sandstone, slate, coarse
    sandstone and diabase. The boundaries between the layers are at
    $z=34$, $z=18$, $z=16$ and $z=8$. We refer to Figure
    \ref{ex3:schematic} for a visual description. The slate layer in
    particular causes this to be a difficult problem to simulate
    numerically due to its hydraulic conductivity being several of orders
    of magnitude smaller than those of the other soils and rocks. The
    initial mesh is aligned with the layers as well as the filter
    locations and the water level in the well. The solution, together with
    a selection of adaptive meshes are given in Figure
    \ref{real_well_4}. The computed flux as a function of degrees of
    freedom is given in Figure \ref{fig:well_4} showing a comparison
    between the mathematical model and the experimental data.

    \begin{figure}[!h]
        \centering 
        \begin{subfigure}{0.45\textwidth}
            \includegraphics[width=\linewidth]{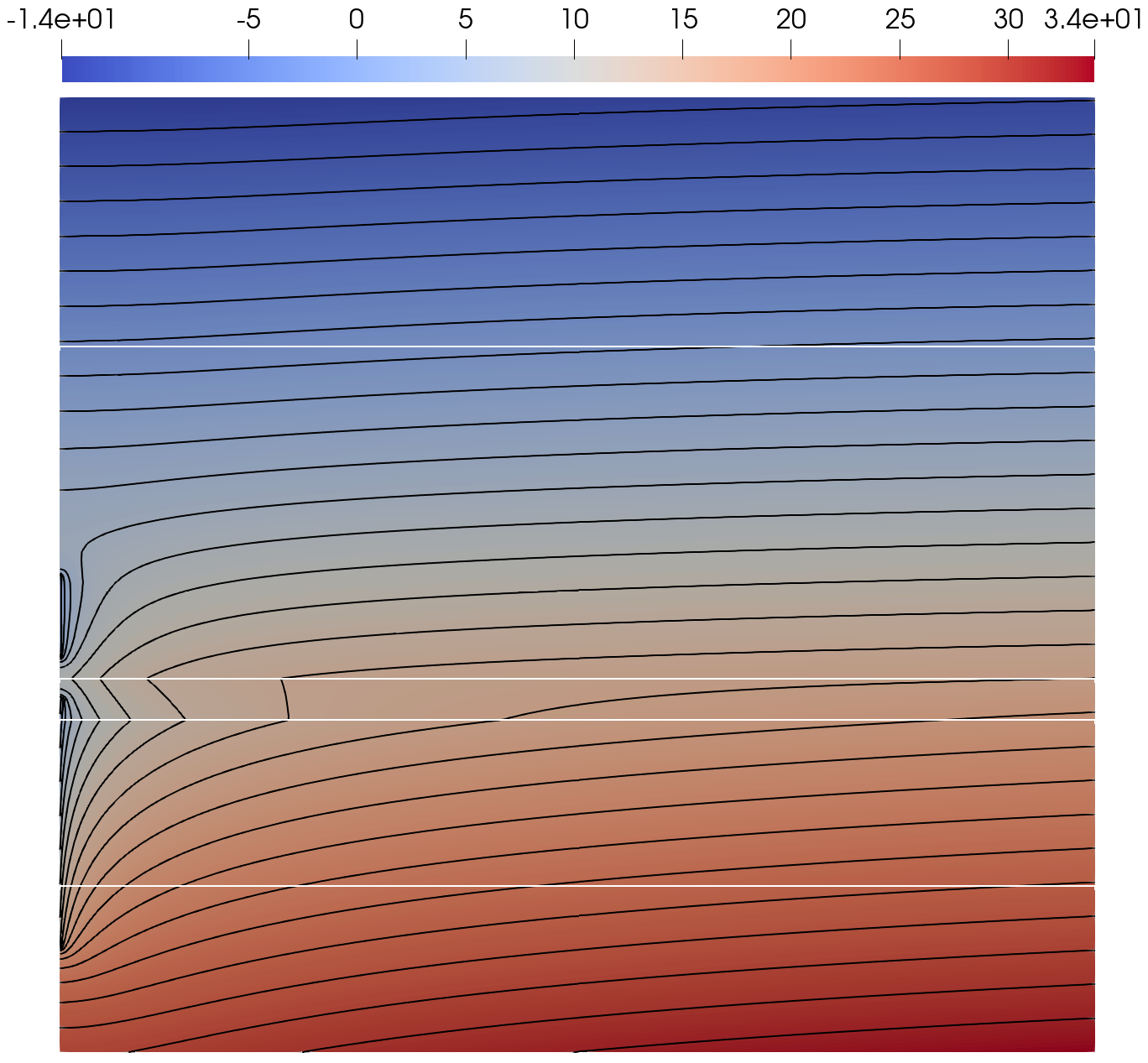}
            \caption{Contours of pressure.}
            \label{well_4}
        \end{subfigure}\hfil 
        \begin{subfigure}{0.45\textwidth}
            \includegraphics[width=\linewidth]{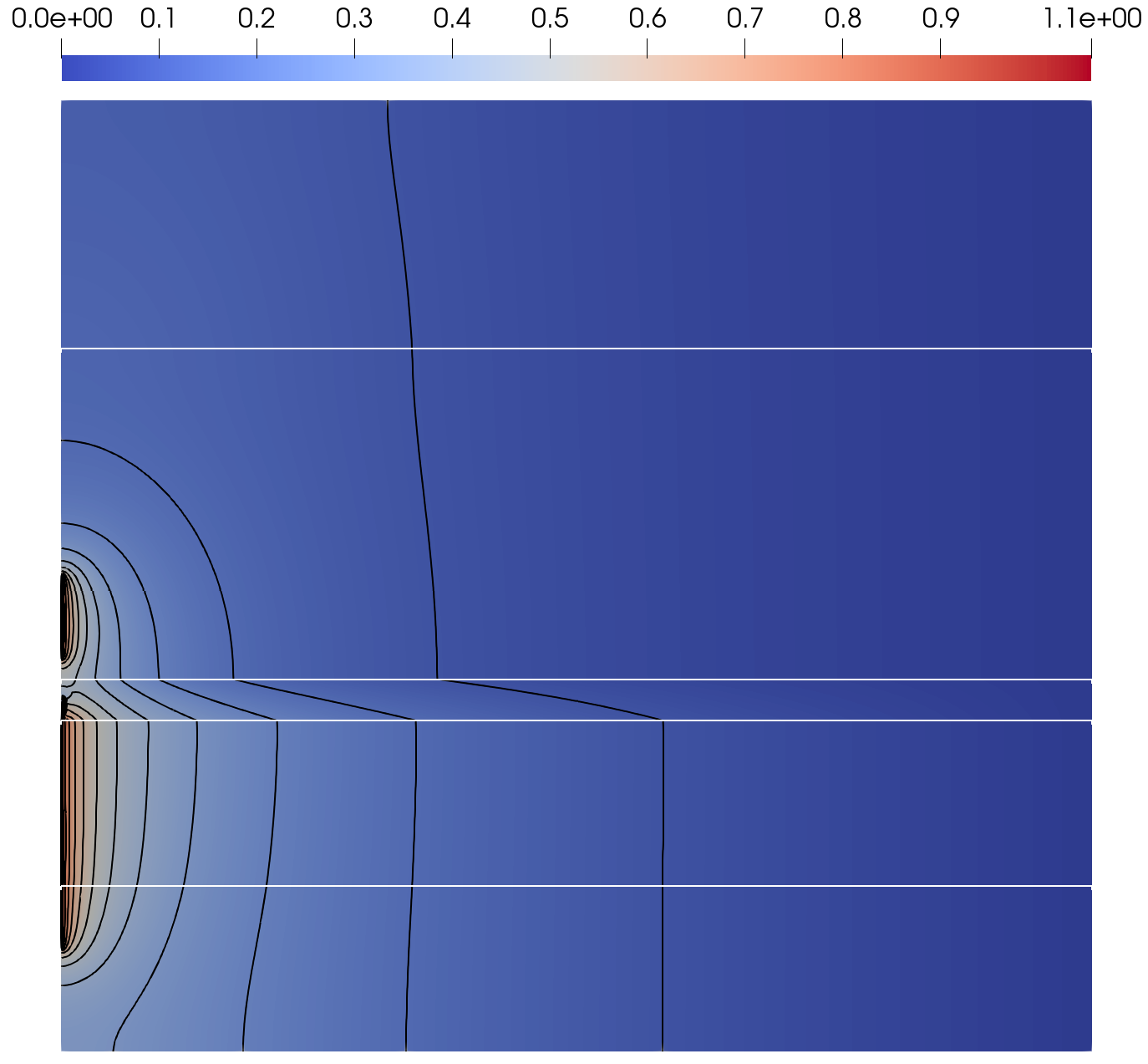}
            \caption{Contours of adjoint pressure, $z_h$.}
            \label{well_dual_4}
        \end{subfigure}\hfil 
        \begin{subfigure}{0.24\textwidth}
            \includegraphics[width=\linewidth]{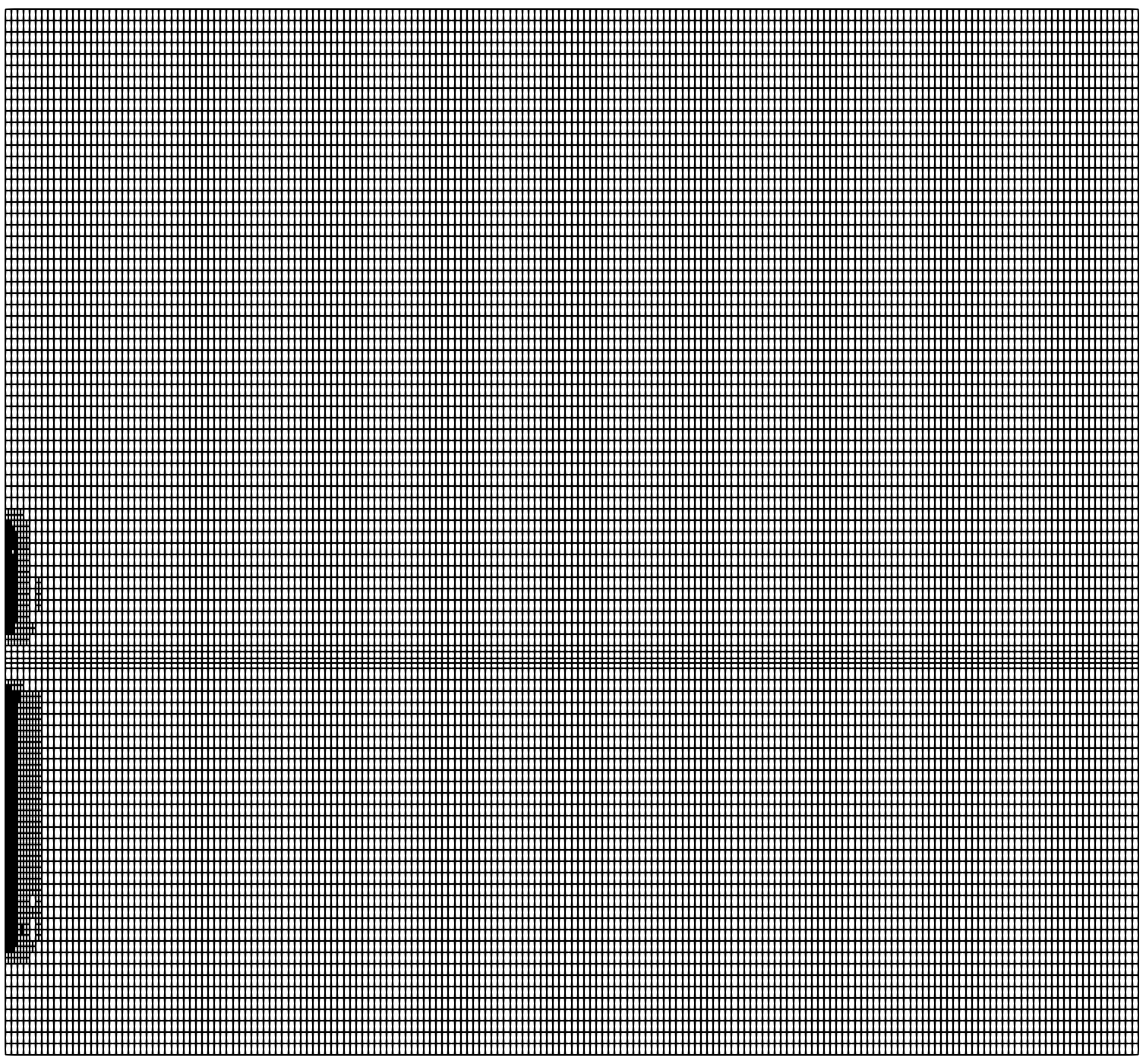}
            \caption{$\mathcal{T}^{20}$}
        \end{subfigure}\hfil 
        \begin{subfigure}{0.24\textwidth}
            \includegraphics[width=\linewidth]{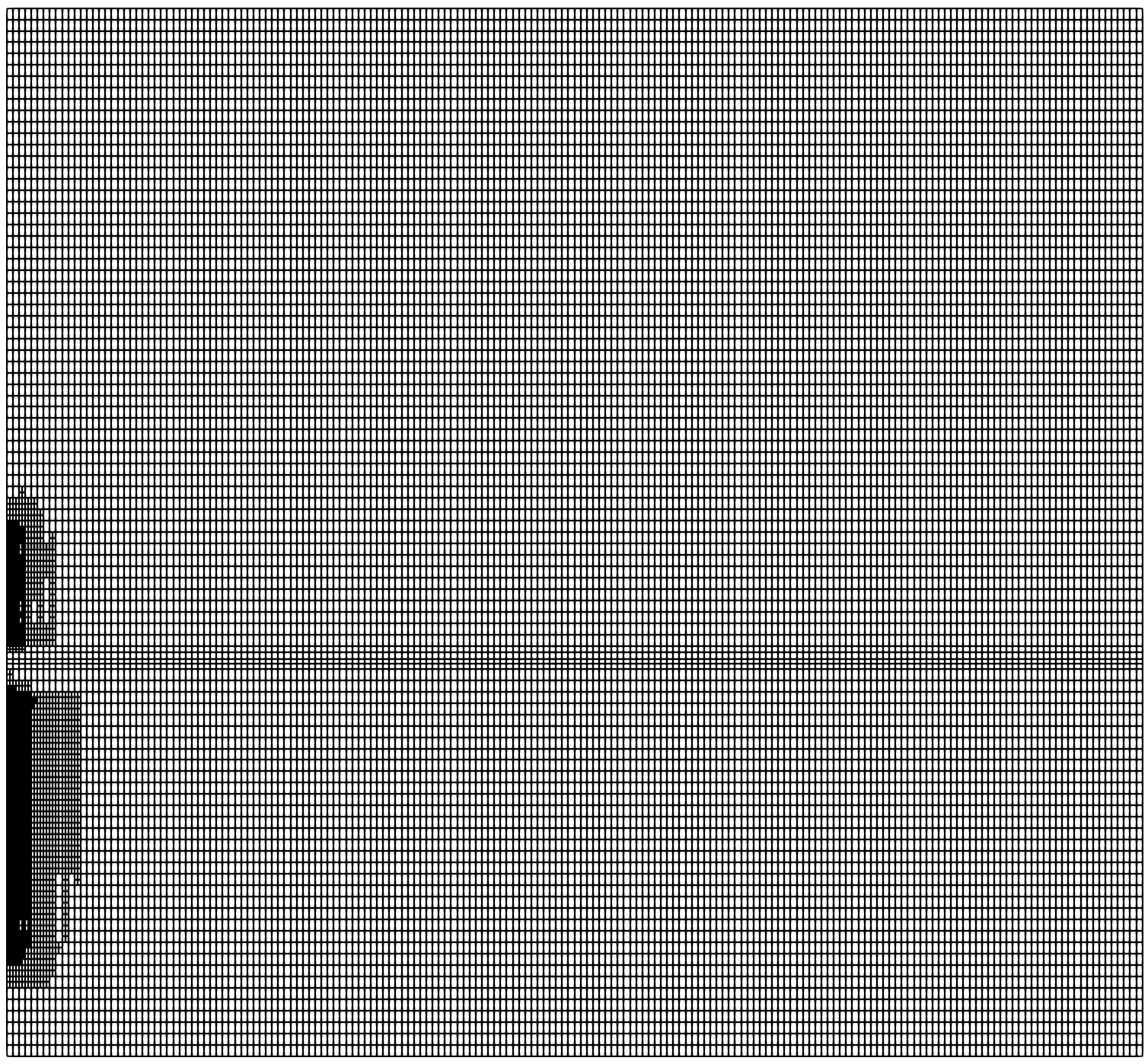}
            \caption{$\mathcal{T}^{25}$}
        \end{subfigure}\hfil 
        \begin{subfigure}{0.24\textwidth}
            \includegraphics[width=\linewidth]{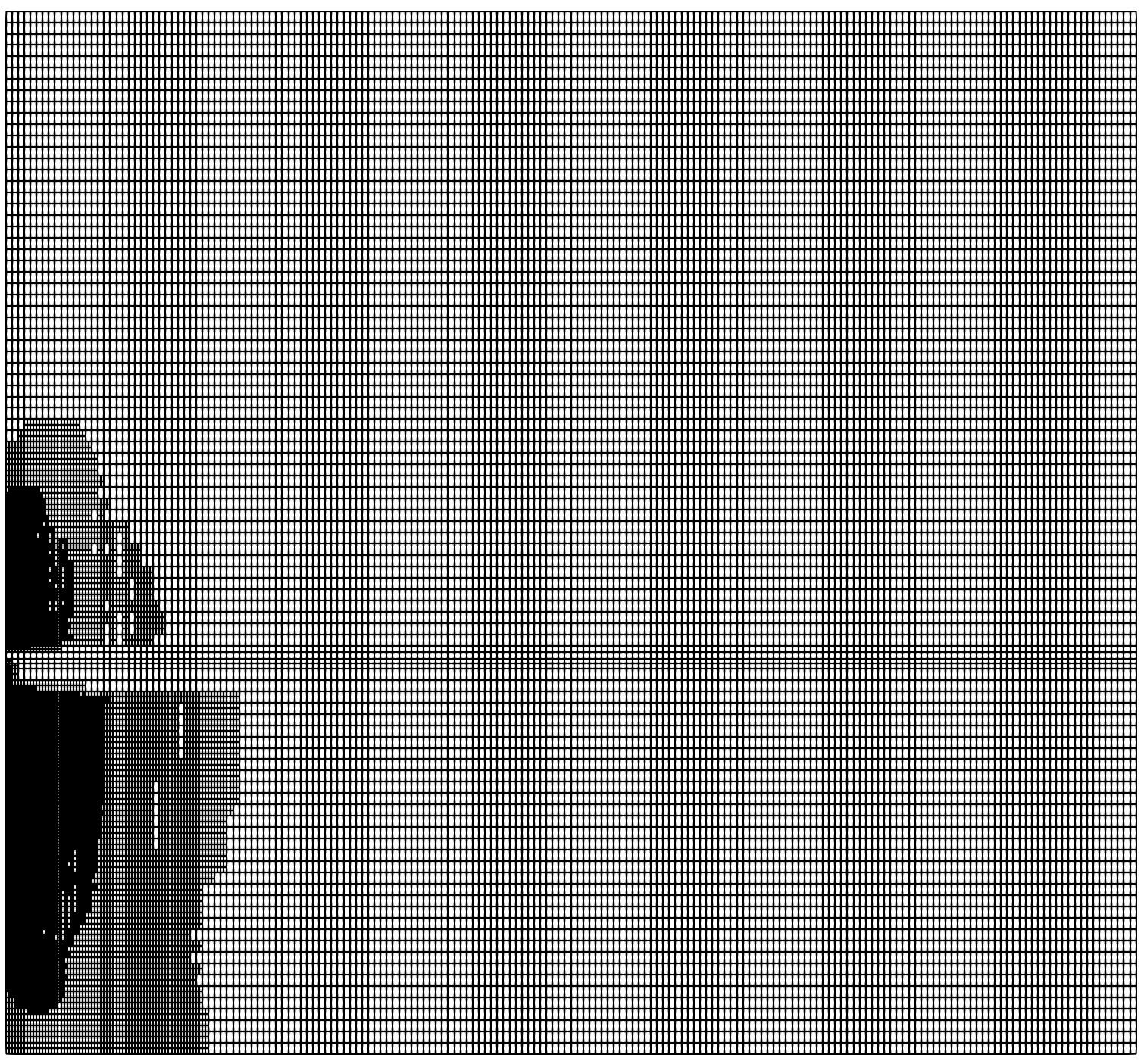}
            \caption{$\mathcal{T}^{35}$}
        \end{subfigure}\hfil 
        \begin{subfigure}{0.24\textwidth}
            \includegraphics[width=\linewidth]{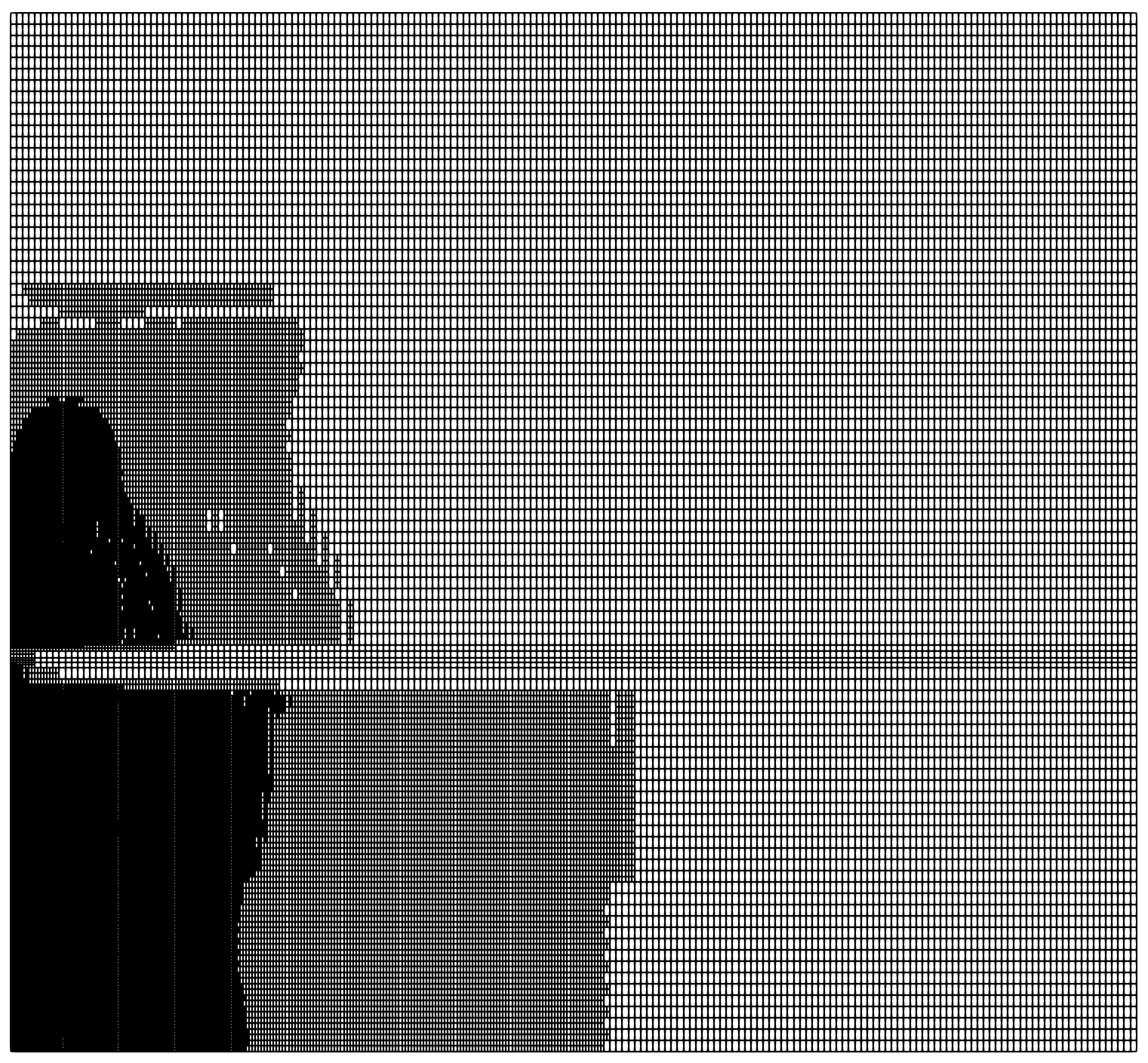}
            \caption{$\mathcal{T}^{45}$}
        \end{subfigure}\hfil 
        \caption{Case study 2, flow through a 5 layer soil. Level set $u =
            0$ marked with red line. The boundaries between the soil layers
            are marked with white lines. See table \ref{parameters} for a
            detailed description of the properties of each layer. Both 
            solutions are represented on  $\mathcal{T}^{45}$ which has 
            approximately 500,000 degrees of freedom. Note that in this
            case the dual problem is much more interesting due to the 
            structure of the inner wall of the well. The meshes appear to 
            show that the soil layers have very different influences on 
            solution accuracy. In particular, the slate layer shows little mesh 
            refinement due to its low permeability relative to the other layers.}
        \label{real_well_4}
    \end{figure}

    \begin{figure}[!h]
        \centering 
        \begin{subfigure}{0.45\textwidth}
            \includegraphics[width=\linewidth]{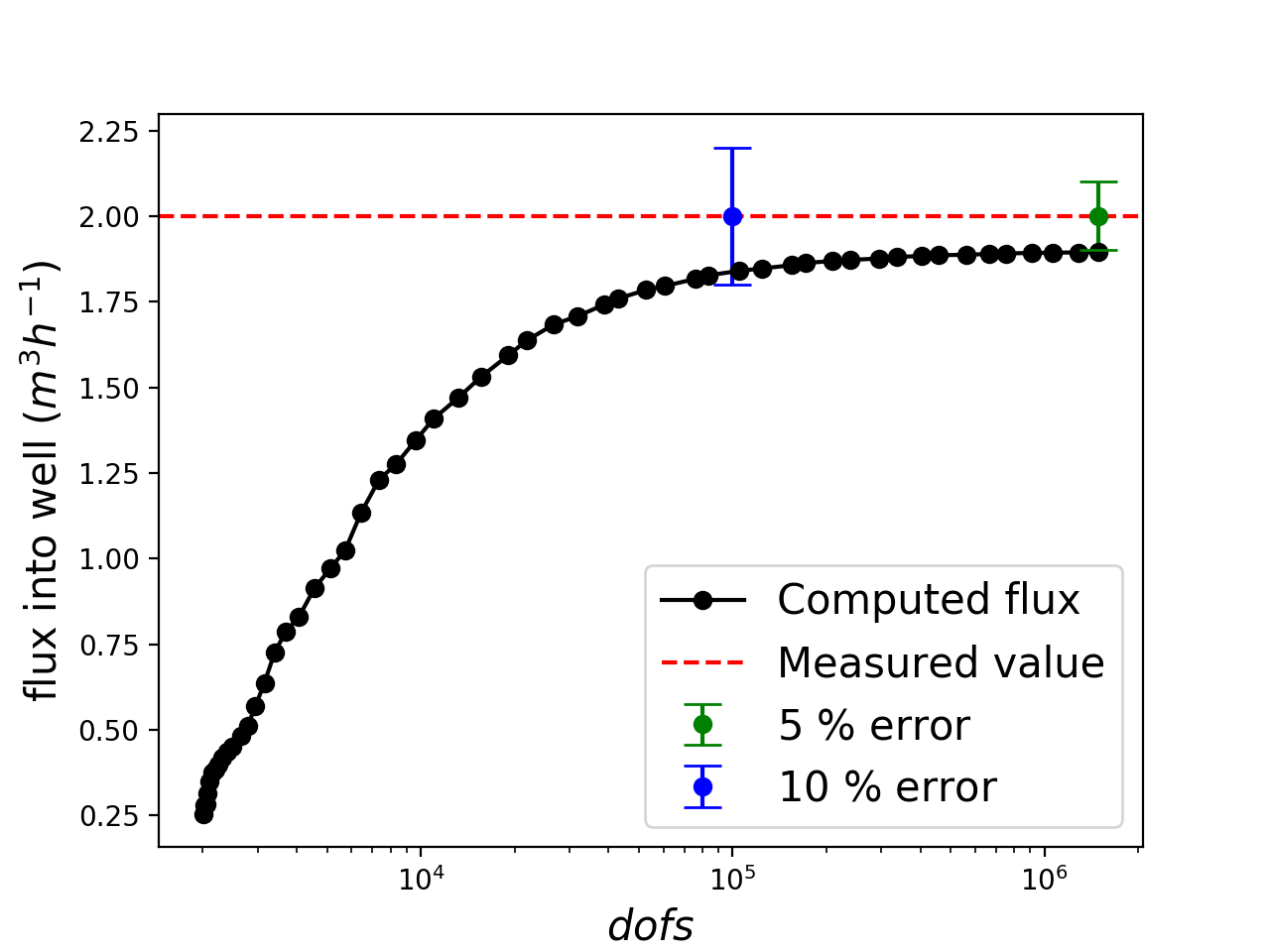}
            \caption{Case study 1. Note that the fully resolved model is
                within a 5\% relative error of the experimental results with
                a 10\% relative error at around $90000$ degrees of
                freedom.}
            \label{fig:well_1}
        \end{subfigure}\hfil 
        \begin{subfigure}{0.45\textwidth}
            \includegraphics[width=\linewidth]{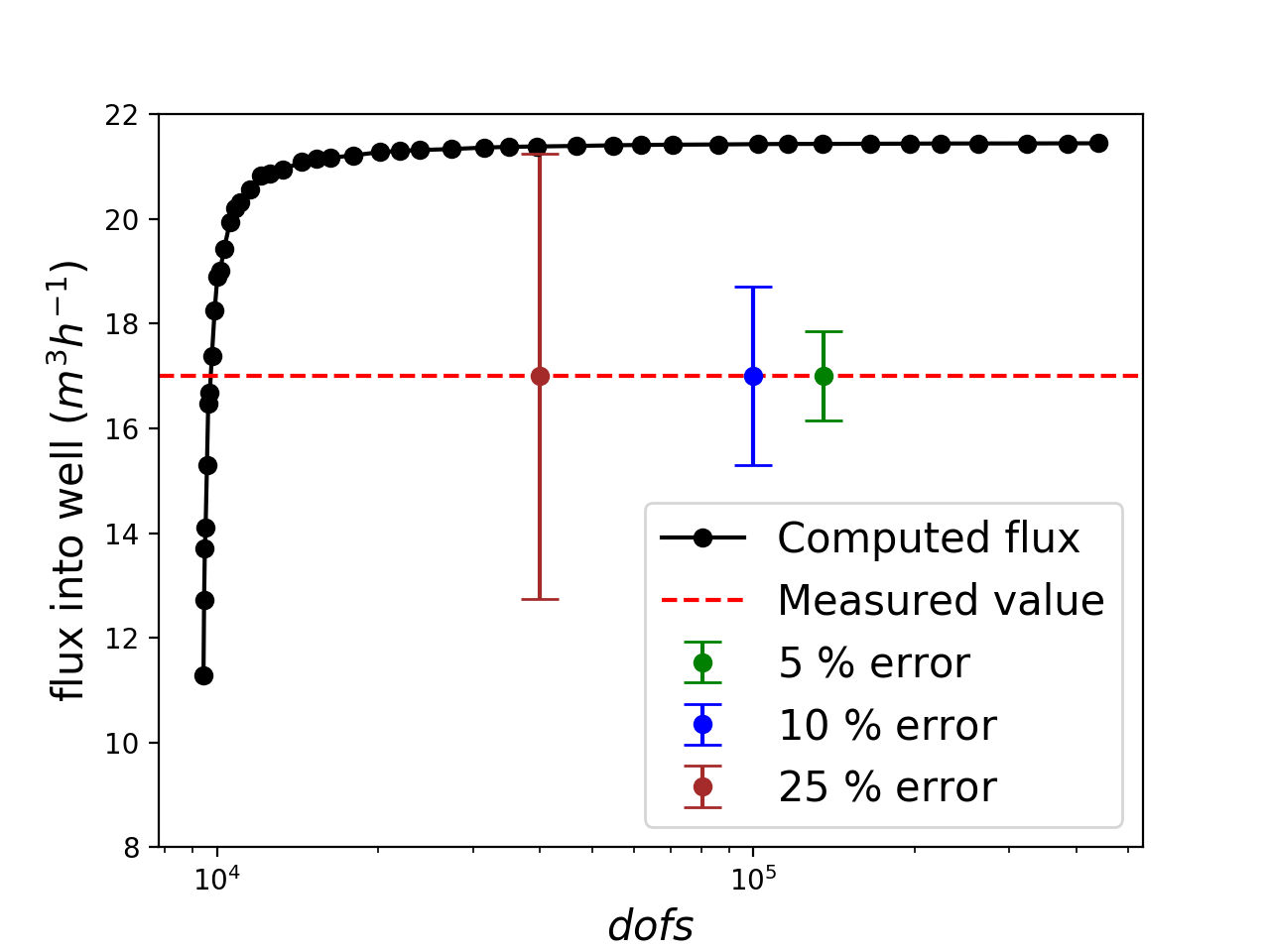}
            \caption{Case study 2. The fully resolved model is around 25\%
                relative error. This is already achieved with $20000$ degrees
                of freedom.
                \phantom{dafuhhadf}
            }
            \label{fig:well_4}
        \end{subfigure}\hfil 
        \caption{Plots displaying the computed value of the water flux
            into the well under successive refinement cycles of the adaptive
            finite element method. This allows to infer the maximal amount
            of water pumped from the well whilst leaving the surrounding
            water table unchanged.}
        \label{wells}
    \end{figure}

    \section{Conclusions \& Discussion} \label{sec:conclusion}
    In this article, we applied techniques from goal-oriented a posteriori
    error estimation to a challenging nonlinear problem involving a groundwater
    flow. For this class of problem, fine uniform meshes do not perform
    well. Indeed, in Figure \ref{orders} we see that convergence can be
    extremely slow on uniform meshes. By comparison, the dual-weighted
    error estimate was shown to perform well under a variety of
    conditions. It has been observed in previous studies (see for example
    \cite{nochetto2008safeguarded}) that due to the approximations that
    must be made to evaluate the error representation numerically, the
    error estimate can perform poorly if the initial mesh in simulations
    is too coarse. In this particular case, we expect that the problem
    originates in the approximation of the dual problem. Since the dual
    solution must satisfy homogeneous Dirichlet boundary conditions on the
    seepage face defined by the primal solution, and since the forcing
    from the quantity of interest is largest here, there is a sharp
    boundary layer at the seepage face which is inevitably poorly resolved
    by a coarse mesh. Notwithstanding, the algorithm produces rapid error
    reduction with effectivity close to 1 once the mesh is sufficiently
    locally refined.  This means that numerical error can be quantified
    with a high degree of confidence, and that the dual-weighted error
    estimate can be used as a termination criterion for an adaptive
    routine.
    
    The case studies most clearly demonstrate the need for adaptive
    techniques in solving problems such as this. The multi-scale nature of
    inhomogeneous soil results in a problem which is extremely challenging
    to solve by conventional numerical methods. Indeed, the error remains
    large on uniform meshes even as the mesh approaches $10^5$ degrees of
    freedom where in the adaptive case a steep and consistent reduction in
    error can be observed with successively refined meshes, see Figure
    \ref{orders}. Applying these robust, computationally efficient methods
    to the case studies allows the accurate quantification of solutions to
    the variational inequality. Note, however, that these case studies are
    still extremely challenging. The assumption of layered soil, for
    example, may not always be physically meaningful. Indeed, we believe
    it is this assumption that affects the performance of case study
    2. For highly variable soils we must use further information, for
    example those provided through resistivity methods. This is the
    subject of ongoing research.

    \section*{Acknowledgement}
    This research was conducted during a visit partially supported through
    the QJMAM Fund for Applied Mathematics.  B.A. was supported through a
    PhD scholarship awarded by the “EPSRC Centre for Doctoral Training in
    the Mathematics of Planet Earth at Imperial College London and the
    University of Reading” EP/L016613/1. T.P. was partially supported
    through the EPSRC grant EP/P000835/1 and the Newton Fund grant
    261865400. CAB acknowledges support from the Conselho Nacional de
    Desenvolvimento Científico e Tecnológico (CNPq) for the Research
    Fellowship Program (grants 300610/2017-3 and 301219/2020-6) and for
    the Research Grant 433481/2018-8.

    \bibliographystyle{unsrt}
    \bibliography{my_collection.bib}

\end{document}